\documentclass[12pt]{amsart}
\usepackage{graphicx}
\usepackage{amsmath}
\usepackage{amsfonts}
\usepackage{amssymb}

\usepackage{setspace}
\usepackage{datetime}
\usepackage[colorlinks,
            linkcolor=black,
            anchorcolor=blue,
            citecolor=black]{hyperref}
\newtheorem{thm}{Theorem}[section]
\newtheorem{cor}[thm]{Corollary}

\newtheorem{lem}[thm]{Lemma}
\newtheorem{prop}[thm]{Proposition}
\theoremstyle{definition}

\theoremstyle{remark}
\newtheorem{rem}[thm]{Remark}
\theoremstyle{conclusion}

\theoremstyle{question}
\numberwithin{equation}{section}

\newcommand{\p}{\partial}
\newcommand{\lee}{\leqslant}
\newcommand{\gee}{\geqslant}
\newcommand{\rt}{\rightarrow}

\begin{document}
\title[Tri-harmonic and bi-harmonic equations with negative exponents]{On properties of positive solutions to nonlinear tri-harmonic and bi-harmonic equations with negative exponents}

\author{Wei Dai, Jingze Fu}

\address{School of Mathematical Sciences, Beihang University (BUAA), Beijing 100191, P. R. China}
\email{weidai@buaa.edu.cn}

\address{School of Mathematical Sciences, Beihang University (BUAA), Beijing 100191, P. R. China}
\email{17377305@buaa.edu.cn}

\thanks{Wei Dai is supported by the NNSF of China (No. 11971049) and the Fundamental Research Funds for the Central Universities.}

\begin{abstract}
In this paper, we investigate various properties (e.g., nonexistence, asymptotic behavior, uniqueness and integral representation formula) of positive solutions to nonlinear tri-harmonic equations in $\mathbb{R}^{n}$ ($n\geq2$) and bi-harmonic equations in $\mathbb{R}^{2}$ with negative exponents. Such kind of equations arise from conformal geometry.
\end{abstract}
\maketitle {\small {\bf Keywords:} Tri-harmonic and bi-harmonic equations; Negative exponents; Nonexistence; Asymptotic behavior; Integral representation formula; Uniqueness; Positive solutions; Doubling lemma. \\

}

\section{Introduction}

\subsection{Background and setting of the problem}
In this paper, we are concerned with the geometrically interesting nonlinear poly-harmonic equations with negative exponents:
\begin{equation}\label{PDE}
  (-\Delta)^{m}u+u^{-q}=0, \quad u>0 \qquad \text{in} \,\,\, \mathbb{R}^{n}
\end{equation}
with $m=2$ or $3$, $n\gee2$ and $q>0$, where $u\in C^{2m}(\mathbb{R}^{n})$.

\smallskip

When $m=1$, second order equations with negative exponents of the type \eqref{PDE} have been intensively studied for both bounded and unbounded domains because of its wide applications to physical models in the study of non-Newtonian fluids, boundary layer phenomena for viscous fluids, chemical heterogenous catalysts, glacial advance, etc., see \cite{A,CRT,SW} and the references therein.

\smallskip

The main purpose of this paper is to investigate various properties (e.g., nonexistence, asymptotic behavior, uniqueness and integral representation formula) of positive smooth solutions to PDEs \eqref{PDE}. Liouville type theorems for fractional order or higher order elliptic equations (i.e., nonexistence of nontrivial nonnegative solutions) in the whole space $\mathbb{R}^n$, in the half space $\mathbb{R}^n_+$ and in general domains $\Omega$ have been extensively studied via the method of moving planes (spheres), the method of scaling spheres and the integral estimates arguments based on Pohozaev identities (see \cite{CD,CDQ0,CDQ,CF,CFY,CGS,CL,CL0,CL1,CLL,CLO,CLM,CLZ,CQ,CY,DL,DLL,DLQ,DP,DPQ,DQ,DQ0,DQ3,DQ2,DQ5,DQ4,DQ1,DQW,DQZ,DW,Farina,F,FW,GNN1,JLX,Li,Lin,LD,LZ,LZ1,Pa,RW,Serrin,WX,ZCCY} and the references therein). These Liouville theorems, in conjunction with the blowing up and re-scaling arguments, are crucial in establishing a priori estimates and hence the existence of positive solutions to non-variational boundary value problems for a class of elliptic equations on bounded domains or on Riemannian manifolds with boundaries (see \cite{CDQ,CL0,DD,GS1,PQS}).

\subsection{Tri-harmonic equations in $\mathbb{R}^{n}$ with $n\gee2$}

We first consider the tri-harmonic equation \eqref{PDE} with $m=3$, which takes the following form:
\begin{equation}\label{PDE-3}
  \Delta^{3}u=u^{-q}, \quad u>0 \qquad \text{in} \,\,\, \mathbb{R}^{n}
\end{equation}
with $n\gee 2$ and $q>0$, where $u\in C^{6}(\mathbb{R}^{n})$.

\smallskip

Equations of the form $(-\Delta)^{m}u+u^{-\frac{4m-k}{k}}=0$ in $\mathbb{R}^{2m-k}$ with $m\gee2$ and $k=1,\cdots,2m-1$ come from the problem of prescribing $Q$-curvature on $\mathbb{S}^{2m-k}$, which is associated with the conformally covariant GJMS operator with the principle part $(-\Delta_{g})^{m}$, discovered by Graham, Jenne, Mason and Sparling in \cite{GJMS} (see also \cite{Juhl,N}). The GJMS operator of order $2m$ with $m\gee2$ is a high-order elliptic operator analogue with the well-known conformal Laplacian in the problem of prescribing scalar curvature,  which is given by
\begin{equation}\label{GJMS}
P_{2m,g_{\mathbf{S}^{n}}}\left(\cdot\right)=\prod_{k=1}^{m}\left[-\Delta_{g_{\mathbf{S}^{n}}}+\left(\frac{n}{2}-k\right)\left(\frac{n}{2}+k-1\right)\right]\left(\cdot\right),
\end{equation}
where the dimension $n\gee3$ and $(\mathbf{S}^{n},g_{\mathbf{S}^{n}})$ is the $n$-sphere equipped with the standard metric $g_{\mathbf{S}^{n}}$. The GJMS operator is conformally covariant in the sense that if we change the standard metric $g_{\mathbf{S}^{n}}$ to its conformal metric $\hat{g}:=v^{\frac{4}{n-2m}}g_{\mathbf{S}^{n}}$ for some positive smooth function $v$ on the $n$-sphere $\mathbf{S}^{n}$, then $P_{2m,g_{\mathbf{S}^{n}}}$ and $P_{2m,\hat{g}}$ satisfy
\begin{equation}\label{GJMS1}
P_{2m,\hat{g}}(\phi)=v^{-\frac{n+2m}{n-2m}}P_{2m,g_{\mathbf{S}^{n}}}(v\phi)
\end{equation}
for any positive smooth function $\phi$ on $\mathbf{S}^{n}$. In particular, if we set $\phi\equiv1$, then \eqref{GJMS1} yields that
\begin{equation}\label{GJMS2}
P_{2m,g_{\mathbf{S}^{n}}}(v)=P_{2m,\hat{g}}(1)v^{\frac{n+2m}{n-2m}}.
\end{equation}
It follows from (1.12) in \cite{Juhl} (see also \cite{N}) that $P_{2m,\hat{g}}(1)=\left(\frac{n}{2}-m\right)Q_{2m,\hat{g}}$ for some scalar curvature quantity $Q_{2m,\hat{g}}$. Consequently, \eqref{GJMS2} implies that
\begin{equation}\label{GJMS4}
P_{2m,g_{\mathbf{S}^{n}}}(v)=\left(\frac{n}{2}-m\right)Q_{2m,\hat{g}}v^{\frac{n+2m}{n-2m}}.
\end{equation}
For any given $k=1,\cdots,2m-1$, let $n=2m-k$. Then we get from \eqref{GJMS4} that
\begin{equation}\label{GJMS3}
P_{2m,g_{\mathbf{S}^{2m-k}}}(v)=-\frac{k}{2}Q_{2m,\hat{g}}v^{-\frac{4m-k}{k}}.
\end{equation}
In order to study the structure of the solution set of \eqref{GJMS3}, let us only take the case when $Q_{2m,\hat{g}}$ is a constant into account. By an appropriate scaling, we may assume that $\frac{k}{2}Q_{2m,\hat{g}}=\pm 1$, and hence \eqref{GJMS3} becomes
\begin{equation}\label{GJMS5}
P_{2m,g_{\mathbf{S}^{2m-k}}}(v)=\mp v^{-\frac{4m-k}{k}}.
\end{equation}
Let us denote by $\pi: \, \mathbf{S}^{2m-k}\rightarrow \mathbb{R}^{2m-k}$ the stereographic projection and define
\begin{equation}\label{GJMS6}
u(x):=v\left(\pi^{-1}(x)\right)\left(\frac{1+|x|^{2}}{2}\right)^{\frac{k}{2}}, \qquad \forall \, x\in\mathbb{R}^{2m-k}.
\end{equation}
By Proposition 1 in \cite{Gra}, we can obtain that
\begin{equation}\label{GJMS7}
\left(\frac{2}{1+|x|^{2}}\right)^{-\frac{4m-k}{2}}(-\Delta)^{m}u(x)=\left(P_{2m,g_{\mathbf{s}^{2m-k}}}(v) \circ \pi^{-1}\right)(x), \qquad \forall \, x\in\mathbb{R}^{2m-k}.
\end{equation}
As a consequence, combining \eqref{GJMS5}, \eqref{GJMS6} and \eqref{GJMS7} yields that
\begin{equation}\label{GJMS8}
(-\Delta)^{m} u=\mp u^{-\frac{4m-k}{k}} \qquad \text{in} \,\,\, \mathbb{R}^{2m-k}.
\end{equation}
By choosing the minus sign in \eqref{GJMS8}, we arrive at
\begin{equation}\label{GJMS9}
(-\Delta)^{m}u+u^{-\frac{4m-k}{k}}=0 \qquad \text{in} \,\,\, \mathbb{R}^{2m-k}
\end{equation}
for every $k=1,\cdots,2m-1$. In particular, when $m=3$, then \eqref{GJMS9} takes the form of equation \eqref{PDE-3}, that is,
\begin{equation}\label{GJMS10}
  \Delta^{3}u=u^{-\frac{12-k}{k}} \qquad \text{in} \,\,\, \mathbb{R}^{6-k}
\end{equation}
for every $k=1,\cdots,5$.

\medskip

On the one hand, the existence and non-existence results, maximum principle type results and asymptotic behavior of radially symmetric solutions for tri-harmonic equations with minus sign:
\begin{equation}\label{PDE-3-}
  \Delta^{3}u=-u^{-q}, \quad u>0 \qquad \text{in} \,\,\, \mathbb{R}^{n}
\end{equation}
have been studied in \cite{DN1,KNS,NNPY}. On the other hand, the existence results, classification results and maximum principle type results for tri-harmonic equations \eqref{PDE-3} with plus sign has been investigated in \cite{FX,KNS,N,NNPY}. In \cite{KNS}, among other things, Kusano, Naito and Swanson established the following existence result via the Schauder-Tychonoff fixed point theorem for both \eqref{PDE-3} and \eqref{PDE-3-}.
\begin{prop}[Theorem 1 in \cite{KNS}]\label{prop0}
Assume $n\gee 3$ and $q>\frac{1}{2}$. Then both the equations \eqref{PDE-3} and \eqref{PDE-3-} possess positive solutions.
\end{prop}
As a special case of Theorem 1 in \cite{KNS}, if $n\gee 3$ and
\begin{equation}\label{g0}
\int_{0}^{+\infty}t\left(1+t^{4}\right)^{-q}dt<+\infty,
\end{equation}
then both the equations \eqref{PDE-3} and \eqref{PDE-3-} possess infinitely many positive radial solutions satisfying the following growth property at $\infty$:
\begin{equation}\label{g1}
C_{1}\left(1+|x|^{4}\right) \leqslant u(x) \leqslant C_{2}\left(1+|x|^{4}\right) \qquad \text { in } \, \mathbb{R}^{n},
\end{equation}
where $C_{1}$ and $C_{2}$ are positive constants. One can easily verify that the integrability condition \eqref{g0} holds if and only if $q>\frac{1}{2}$. In \cite{NNPY}, among other things, Ng\^{o}, Nguyen, Phan and Ye also established the following existence result for radial solutions to equation \eqref{PDE-3}.
\begin{prop}[Proposition 4.5 in \cite{NNPY}]\label{prop1}
For any $n\gee 1$ and $-1\lee q<+\infty$, the equation \eqref{PDE-3} has infinitely many positive radial solutions.
\end{prop}

\medskip

For the classification of positive solutions to the corresponding integral equations associated with \eqref{PDE}, we refer the reader to \cite{FX,HW,Li,Liu,N,Xu,Xu1}. In fact, when $m=3$, it was already proved by Feng and Xu in \cite{FX} that, suppose $u\in C^{6}(\mathbb{R}^{5})$ is a positive solution to
\begin{equation}\label{e6}
  u(x)=\frac{1}{64\pi^{2}}\int_{\mathbb{R}^{5}}|x-y|u^{-q}(y)dy,
\end{equation}
then $q=11$, and $u$ must assume the unique form $u(x)=c(1+|x|^{2})^{\frac{1}{2}}$ up to dilations and translations. For more literature on poly-harmonic and fractional order problems (with negative exponents), please refer to \cite{AMR,BR,CX,DFG,DN,FKT,Gue,GW,GW1,HW,KR,Lai1,Lai2,LY,MR,MW,MW1,MX,N0,N,NNP,NNP1,NX,NX1,WY} and the references therein.

\bigskip

We can prove various properties (including nonexistence, asymptotic behavior, uniqueness and integral representation formula) of positive entire solutions to the tri-harmonic equation \eqref{PDE} with negative exponents, that is, $m=3$ and $n\gee2$ in \eqref{PDE}.

\begin{thm}\label{thm1}
Assume $m=3$, $n\gee2$ and $q>0$. Then, we have \\
(i) \, For any $n\gee2$, \eqref{PDE} admits no positive entire solution $u\in C^{6}(\mathbb{R}^{n})$ satisfying $u=o(|x|^{4})$ as $|x|\rightarrow+\infty$ if $0<q<2$. \\
(ii) \, For $n=2,3,4$ and any $q>0$, \eqref{PDE} admits no positive entire solution $u\in C^{6}(\mathbb{R}^{n})$ satisfying $u=o(|x|^{4})$ as $|x|\rightarrow+\infty$. \\
(iii) \, For any $n\gee2$ and $q>0$, suppose $u$ is a positive $C^{6}$ entire solution to \eqref{PDE} satisfying $u=o(|x|^{4})$ at $\infty$, then $\bar{u}(r)\gee cr^{\frac{4}{q}}$ for some positive constant $c>0$ and any $r\gee0$, where $\bar{u}(r)$ denotes the spherical average of $u$ over $\partial B_{r}(0)$. Consequently, if $q\gee 2$, then \eqref{PDE} admits no positive entire solution $u\in C^{6}(\mathbb{R}^{n})$ satisfying $u=o\left(|x|^{\frac{4}{q}}\right)$ as $|x|\rightarrow+\infty$. Moreover, if $q\gee 2$, then \eqref{PDE} admits no positive radially symmetric entire solution $u\in C^{6}(\mathbb{R}^{n})$ satisfying $u=o(r^{4})$ at $\infty$ and
\begin{equation}\label{e8}
  \liminf_{r\rightarrow+\infty}\frac{u(r)}{r^{\frac{4}{q}}}=0.
\end{equation}
(iv) \, For any $q\gee 2$, suppose $u$ is a positive $C^{6}$ entire solution to \eqref{PDE} satisfying $u=o(|x|^{4})$ at $\infty$, then $\bar{u}(r)\lee Cr^{2}$ for some constant $C>0$ and any $r\gee 1$ when $n\gee 2$, and $\bar{u}(r)\gee cr$ for some positive constant $c>0$ and all $r\gee 0$ when $n=5$. Moreover, if $u$ is a positive radially symmetric $C^{6}$ entire solution to \eqref{PDE} satisfying $u=o(r^{4})$ at $\infty$, then $u(r)\lee Cr^{2}$ for any $r:=|x|\gee 1$ when $n\gee 2$, and $u(r)\gee cr$ for all $r:=|x|\gee 0$ when $n=5$, that is, $\limsup\limits_{|x|\rightarrow+\infty}\frac{u(x)}{|x|^{2}}\in[0,+\infty)$ for every $n\gee 2$ and $\liminf\limits_{|x|\rightarrow+\infty}\frac{u(x)}{|x|}\in(0,+\infty]$ for $n=5$. \\
(v) \, For any $n\gee2$ and $q>0$, \eqref{PDE} admits no positive entire solution $u\in C^{6}(\mathbb{R}^{n})$ satisfying $u=o\left(|x|^{\frac{6}{q+1}}\right)$ as $|x|\rightarrow+\infty$. Moreover, \eqref{PDE} admits no positive radially symmetric entire solution $u\in C^{6}(\mathbb{R}^{n})$ satisfying
\begin{equation}\label{a24}
  \limsup_{r\rightarrow+\infty}\frac{u(r)}{r^{\frac{6}{q+1}}}<+\infty \qquad \text{and} \qquad \liminf_{r\rightarrow+\infty}\frac{u(r)}{r^{\frac{6}{q+1}}}=0.
\end{equation}
(vi) \, For any $n\gee2$ and $q>0$, if $u$ has at most linear growth (uniformly) at $\infty$, that is,
\begin{equation}\label{e9}
  \lim _{|x|\rightarrow+\infty}\frac{u(x)}{|x|}=\alpha\in[0,+\infty),
\end{equation}
then $q\gee 5$. In addition, we have $\alpha>0$ if $n=5$. Furthermore, if $q=11$, then, up to dilations and translations, $u$ must assume the unique form: $u(x)=c(1+|x|^{2})^{\frac{1}{2}}$. \\
(vii) \, For any $n\gee5$ and $q>5$, if $u\in C^{6}(\mathbb{R}^{n})$ is a positive radially symmetric entire solution to \eqref{PDE} such that $\liminf\limits_{|x|\rightarrow+\infty}\frac{u(x)}{|x|}\in(0,+\infty]$ and $\lim\limits_{r\rightarrow+\infty}u''(r)=\rho\in\mathbb{R}$, then $\rho\gee 0$. If $\rho>0$, then $u$ has exactly quadratic growth at $\infty$, that is,
\begin{equation}\label{e1}
  \lim _{|x| \rightarrow+\infty}\frac{u(x)}{|x|^{2}}=\frac{\rho}{2}\in(0,+\infty).
\end{equation}
If $\rho=0$, then $u$ has exactly linear growth at $\infty$, that is,
\begin{equation}\label{a31}
  \lim _{|x| \rightarrow+\infty}\frac{u(x)}{|x|}=
  -\frac{1}{n-1}\int_{0}^{+\infty}t^{2}\left(\Delta^{2}\bar{u}(t)-\frac{n-3}{t}\left(\Delta\bar{u}\right)'(t)\right)dt\in(0,+\infty).
\end{equation}
(viii) \, If $n=5$, $q\gee 2$ and $u\in C^{6}(\mathbb{R}^{5})$ is a positive entire solution to \eqref{PDE} satisfying $u=o(|x|^{4})$ as $|x|\rightarrow+\infty$, then
\begin{equation}\label{a32-0}
  \bar{u}''(r)>0, \quad \bar{u}'''(r)<0, \qquad \forall \, r>0.
\end{equation}
Moreover, if assume further $q>5$ and $u$ is radially symmetric, then either $u$ has exactly quadratic growth at $\infty$ or $u$ has exactly linear growth at $\infty$. \\
(ix) \, If $n=5$, $q>6$ and the positive entire solution $u$ has exactly linear growth at $\infty$, then $u$ has the following integral representation:
\begin{equation}\label{e5}
  u(x)=\frac{1}{64\pi^{2}}\int_{\mathbb{R}^{5}}|x-y|u^{-q}(y)dy+\gamma,
\end{equation}
where $\gamma$ is a constant. Moreover, $\gamma<0$ if $6<q<11$, $\gamma=0$ if $q=11$ and $\gamma>0$ if $q>11$. \\
(x) \, Assume $n\gee 3$ but $n\neq4, 6$, $q>\frac{1}{2}$ and $u$ is a radially symmetric positive entire solution to \eqref{PDE} such that $\liminf\limits_{|x|\rightarrow+\infty}\frac{u(x)}{|x|^{4}}\in(0,+\infty]$, then $\lim\limits_{|x|\rightarrow+\infty}\frac{u(x)}{|x|^{4}}=\frac{1}{8n(n+2)}\lim\limits_{|x|\rightarrow+\infty}\Delta^{2}u(x)\in(0,+\infty)$. Conversely, for any $n\gee 2$ and $q>0$, suppose $u$ is a radially symmetric positive entire solution to \eqref{PDE}. We have: (a) If $\lim\limits_{|x|\rightarrow+\infty}\Delta^{2}u(x)\in(0,+\infty]$, then $\liminf\limits_{|x|\rightarrow+\infty}\frac{u(x)}{|x|^{4}}\in(0,+\infty]$. Furthermore, if assume further $n\gee 3$ but $n\neq4, 6$ and $q>\frac{1}{2}$, then $\lim\limits_{|x|\rightarrow+\infty}\frac{u(x)}{|x|^{4}}=\frac{1}{8n(n+2)}\lim\limits_{|x|\rightarrow+\infty}\Delta^{2}u(x)\in(0,+\infty)$. (b) If $\lim\limits_{|x|\rightarrow+\infty}\Delta^{2}u(x)=0$, then $u$ satisfy the sub poly-harmonic property and hence has at most quadratic growth at $\infty$.
\end{thm}

\begin{rem}\label{rem1}
In Theorem \ref{thm1}, we only need the assumption $u=o(|x|^{4})$ as $|x|\rightarrow+\infty$ to guarantee the sub poly-harmonic property ``$\Delta^{2}u<0$ and $\Delta u>0$ in $\mathbb{R}^{n}$" for positive entire solutions to tri-harmonic equations \eqref{PDE} (see Lemma \ref{lem5}). In Theorem 1 in \cite{N}, Ng\^{o} need the assumption
\begin{equation}\label{a0}
  \limsup_{|x|\rightarrow+\infty}\frac{u(x)}{|x|^{4}}\lee0 \quad \text{and} \quad \limsup_{|x|\rightarrow+\infty}\frac{-u(x)}{|x|^{2}}\lee0
\end{equation}
in order to derive the sub poly-harmonic property for $C^{6}$ solution $u$ to the differential inequality $(-\Delta)^{3}u<0$ in $\mathbb{R}^{n}$. For positive solution $u$ to tri-harmonic equations \eqref{PDE} with negative exponents, the condition \eqref{a0} is equivalent to our assumption $u=o(|x|^{4})$ as $|x|\rightarrow+\infty$. The upper bound estimate \eqref{7-3a} also indicates that the assumption $u=o(|x|^{4})$ at $\infty$ is necessary in the sense that, if the sub poly-harmonic property holds for solution $u$, then $u$ has no more than quadratic growth at $\infty$ in the sense of spherical average, i.e., $\bar{u}(r)\lee Cr^{2}$ for $r$ sufficiently large. In particular, if $u$ is a radially symmetric positive entire solution satisfying the sub poly-harmonic property, then $u$ has no more than quadratic growth at $\infty$. Throughout the results in Theorem \ref{thm1} and in Section 3, the assumption ``$u=o(|x|^{4})$ at $\infty$" can be replaced by the sub poly-harmonic property. One should note that positive $C^{6}$ entire solution $u$ with exactly linear growth at $\infty$ in $\mathbb{R}^{5}$, exactly quadratic growth at $\infty$ in $\mathbb{R}^{4}$ and exactly cubic growth at $\infty$ in $\mathbb{R}^{3}$ satisfies $u=o(|x|^{4})$ as $|x|\rightarrow+\infty$.
\end{rem}

\begin{rem}\label{rem2}
Suppose $u\in C^{6}(\mathbb{R}^{n})$ positive radially symmetric solution to \eqref{PDE}, then the assumption $\lim\limits_{r\rightarrow+\infty}u''(r)=\rho\in\mathbb{R}$ in (vii) in Theorem \ref{thm1} can be deduced from ``$u''(r)>0$ and $u'''(r)<0$ for any $r\in(0,+\infty)$". When $n=5$, for general positive entire solution $u$, the property ``$\bar{u}''(r)>0$ and $\bar{u}'''(r)<0$ for any $r\in(0,+\infty)$" can be deduced from the sub poly-harmonic property or the assumption ``$u=o(|x|^{4})$ at $\infty$" (see Lemma \ref{lem10} in Section 3).
\end{rem}

\begin{rem}\label{rem9}
From $(i)$ and $(ii)$ Theorem \ref{thm1}, we can see that if \eqref{PDE} possesses a $C^{6}$ positive entire solution $u$ satisfying $u=o(|x|^{4})$ at $\infty$ or the sub poly-harmonic property, then $q\gee 2$ and $n\gee 5$. Now we assume $u\in C^{6}(\mathbb{R}^{n})$ is a radially symmetric positive entire solution to \eqref{PDE}. Suppose $q\gee 2$, $n\gee 5$ and $u=o(r^{4})$ at $\infty$, by $(iii)$, $(iv)$ and $(v)$ in Theorem \ref{thm1}, we have, $u$ must has exactly quadratic growth at $\infty$ if $q=2$, $u$ must has more than linear but at most quadratic growth at $\infty$ if $2<q<5$, $u$ must has at least linear but at most quadratic growth at $\infty$ if $n=5$ and $q\gee 5$. Suppose $q>0$ and $n\gee 2$, $(v)$ in Theorem \ref{thm1} implies that, $u$ has more than quintic growth at $\infty$ if $q<\frac{1}{5}$, $u$ has more than quartic growth at $\infty$ if $q<\frac{1}{2}$, $u$ has more than linear growth at $\infty$ if $q<5$. Suppose $q>5$, $n=5$ and $u=o(r^{4})$ at $\infty$, from $(viii)$ in Theorem \ref{thm1}, we derive that either $u$ has exactly quadratic growth at $\infty$ or $u$ has exactly linear growth at $\infty$. Suppose $n\gee 3$ but $n\neq4, 6$, $q>\frac{1}{2}$ and $u$ has at least quartic growth at $\infty$, $(x)$ in Theorem \ref{thm1} yields that $u$ must has exactly quartic growth at $\infty$.
\end{rem}

\subsection{2D bi-harmonic equation}

Next, in equation \eqref{PDE}, we will also consider the bi-harmonic analogue of the tri-harmonic case in $\mathbb{R}^{2}$, i.e., $m=2$ and $n=2$.

\smallskip

Let us briefly review the background of the bi-harmonic equations \eqref{PDE} arising from conformal geometry. Suppose $(M,g)$ is a smooth $4$-dimensional Riemannian manifold, Paneitz \cite{P} extended the Laplace-Beltrami operator $\Delta_{g}$ (which is conformally covariant on $2$-dimensional Riemannian manifold) to the fourth order operator $P^{4}_{g}$ on $(M,g)$ which is defined by
\begin{equation}\label{P1}
  P^{4}_{g}:=\Delta_{g}^{2}-\delta\left[\left(\frac{2}{3}R_{g}g-2Ric_{g}\right)d\right],
\end{equation}
where $\delta$ is the divergent operator, $R_{g}$ is the scalar curvature of $g$ and $Ric_{g}$ si the Ricci curvature of $g$. The Paneitz operator $P^{4}_{g}$ also has the conformally covariant property, that is, if $\tilde{g}:=e^{2u}g$, then $P^{4}_{\tilde{g}}=e^{-4u}P^{4}_{g}$. A generalization of the Paneitz operator $P^{4}_{g}$ to manifolds of other dimensions $n\gee3$, due to Branson \cite{Branson}, is given by
\begin{equation}\label{P2}
  P^{n}_{g}:=\Delta_{g}^{2}-\delta\left[\left(\frac{(n-2)^{2}+4}{2(n-1)(n-2)}R_{g}g-\frac{4}{n-2}Ric_{g}\right)d\right]+\frac{n-4}{2}Q^{n}_{g},
\end{equation}
where
\begin{equation}\label{Q1}
Q^{n}_{g}=-\frac{1}{2(n-1)}\Delta_{g}R_{g}+\frac{n^{3}-4n^{2}+16n-16}{8(n-1)^{2}(n-2)^{2}}R_{g}^{2}-\frac{2}{(n-2)^{2}}|Ric_{g}|^{2}.
\end{equation}
One should note that $P^{n}_{g}$ reduces to $P^{4}_{g}$ when $n=4$.

\smallskip

Similar to the conformal Laplace-Beltrami operator $L^{n}_{g}:=-\Delta_{g}+\frac{n-2}{4(n-1)}R_{g}$ for $n\gee2$, the operator $P^{n}_{g}$ has conformal property: if $n\neq4$ and $\tilde{g}:=u^{\frac{4}{n-4}}g$ is a conformal metric of $g$, then $P^{n}_{g}(\phi u)=P^{n}_{\tilde{g}}(\phi)u^{\frac{n+4}{n-4}}$ for any $\phi\in C^{\infty}(M)$. In particular, if $\phi=1$, then $P^{n}_{g}(\phi u)=\frac{n-4}{2}Q^{n}_{\tilde{g}}(\phi)u^{\frac{n+4}{n-4}}$.

\smallskip

As a consequence, in the Euclidean case, suppose that $u>0$ and $g_{0}$ is a standard flat metric on $\mathbb{R}^{n}$, and let $g:=u^{\frac{4}{n-4}}g_{0}$ be a conformal metric of $g_{0}$, then $u$ satisfies
\begin{equation}\label{GPDE}
\Delta^{2}u=\frac{n-4}{2}Q^{n}_{g}u^{\frac{n+4}{n-4}} \qquad \text{in} \,\, \mathbb{R}^{n},
\end{equation}
where $n\lee3$ or $n\gee5$. If $n\gee5$, the conformal metric $g$ satisfies $Q^{n}_{g}=cu^{p-\frac{n+4}{n-4}}$ ($c>0$, $p>0$, $u>0$ is the conformal factor), then the classification of positive solutions to \eqref{GPDE} has been established in Lin \cite{Lin} (see also Wei and Xu \cite{WX}). If $n\lee3$, in order to find conformal metric $g$ on $\mathbb{R}^{n}$ satisfying $Q^{n}_{g}=cu^{-q-\frac{n+4}{n-4}}$ ($c>0$, $q>0$, $u>0$ is the conformal factor), we need to consider the following equation with negative exponents (i.e., equation \eqref{PDE} with $m=2$):
\begin{equation}\label{DE}
\Delta^{2}u+u^{-q}=0, \quad u>0 \qquad \text{in} \,\, \mathbb{R}^{n}.
\end{equation}

\smallskip

When $n=3$, the existence and asymptotic behavior of positive solutions to 3D equation \eqref{DE} have been studied by Choi and Xu \cite{CX}, Duoc and Ng\^{o} \cite{DN}, Guerra \cite{Gue}, Hyder and Wei \cite{HW}, Lai \cite{Lai1,Lai2}, McKenna and Reichel \cite{MR} and Xu \cite{Xu}.
\begin{thm}[\cite{CX,DN,Gue,HW,Lai1,Lai2,MR,Xu}]\label{known results}
Assume that $n=3$ in equation \eqref{DE}. Then \\
(i)\, There is no $C^{4}$ positive entire solution to \eqref{DE} if $0<q\lee1$. \\
(ii)\, If $u\in C^{4}(\mathbb{R}^{3})$ has at most linear growth (uniformly) at $\infty$, that is,
\begin{equation}\label{e0}
  \lim _{|x| \rightarrow+\infty}\frac{u(x)}{|x|}=\alpha\in[0,+\infty),
\end{equation}
then $\alpha>0$ and $q>3$. Moreover, if $q=7$, then, up to dilations and translations, $u$ must assume the unique form: $u(x)=\sqrt{\frac{1}{\sqrt{15}}+|x|^{2}}$. \\
(iii)\, If $q>3$, then for any $\alpha>0$, there exists a unique radially symmetric solution $u\in C^{4}(\mathbb{R}^{3})$ to \eqref{DE} (with exactly linear growth) such that
\begin{equation}\label{e2}
  \lim_{|x|\rightarrow+\infty}\frac{u(x)}{|x|}=\alpha>0.
\end{equation}
(iv)\, If $q>1$, there exists a radially symmetric solution $u\in C^{4}(\mathbb{R}^{3})$ to \eqref{DE} with exactly quadratic growth, that is,
\begin{equation}\label{e1}
  \lim _{|x| \rightarrow+\infty}\frac{u(x)}{|x|^{2}}=C>0.
\end{equation}
(v)\, If $q>3$, then any radially symmetric solution to \eqref{DE} is either exactly linear growth or exactly quadratic growth. \\
(vi)\, If $1<q<3$, there exists a unique radially symmetric solution $u\in C^{4}(\mathbb{R}^{3})$ to \eqref{DE} such that
\begin{equation}\label{e3}
  \lim_{|x|\rightarrow+\infty} |x|^{-\frac{4}{q+1}}u(x)=K_{q}^{-\frac{1}{q+1}},
\end{equation}
where $K_{q}=\tau(2-\tau)(\tau+1)(\tau-1)$ and $\tau=\frac{4}{q+1}$. \\
(vii)\, If $q=3$, then there exists a unique radially symmetric solution $u\in C^{4}(\mathbb{R}^{3})$ to \eqref{DE} such that
\begin{equation}\label{e4}
  \lim_{|x|\rightarrow+\infty}\frac{u(x)}{|x|\log(|x|)^{\frac{1}{4}}}=2^{\frac{1}{4}}.
\end{equation}
(viii)\, If $q>4$ and $u$ has exactly linear growth at $\infty$, then $u$ has the following integral representation:
\begin{equation}\label{e5}
  u(x)=\frac{1}{8\pi}\int_{\mathbb{R}^{3}}|x-y|u^{-q}(y)dy+\gamma,
\end{equation}
where $\gamma$ is a constant. Moreover, $\gamma<0$ if $4<q<7$, $\gamma=0$ if $q=7$ and $\gamma>0$ if $q>7$.
\end{thm}

\smallskip

One should note that the 3D bi-harmonic equation \eqref{PDE} is essentially similar to the 5D tri-harmonic case. In fact, when $m=2$, it was already proved by Xu in \cite{Xu} that, suppose $u\in C^{4}(\mathbb{R}^{3})$ is a positive solution to
\begin{equation}\label{e7}
  u(x)=\frac{1}{8\pi}\int_{\mathbb{R}^{3}}|x-y|u^{-q}(y)dy,
\end{equation}
then $q=7$, and $u$ must assume the unique form $u(x)=c(1+|x|^{2})^{\frac{1}{2}}$ up to dilations and translations. Thus we may expect the absence of positive entire solution to the bi-harmonic equation \eqref{PDE} in two dimension case $n=2$.

\medskip

In this paper, we will also investigate the bi-harmonic equation \eqref{PDE} in $\mathbb{R}^{2}$. Being essentially different from the abundant existence results of various positive entire solutions with different asymptotic behaviors at $\infty$ in the 3D case (see Theorem \ref{known results}), we can prove the nonexistence of positive entire solutions to the planar bi-harmonic equation \eqref{PDE} with negative exponents, that is, $n=2$ and $m=2$ in \eqref{PDE}.

\begin{thm}\label{thm0}
Assume $n=2$ and $m=2$. Then, for any $q>0$, there is no positive entire solution $u\in C^{4}(\mathbb{R}^{2})$ to \eqref{PDE}.
\end{thm}

\begin{rem}\label{rem0}
Theorem 3.1 in Mckenna and Reichel \cite{MR} implies that, for $n\gee4$ the equation $\Delta^{2}u=-u^{-q}$ in $\mathbb{R}^{n}$ has non-radial positive entire solutions given by $u(x',x_{n})=v(|x'|)$, where $v(r)$ is a radial positive entire solution satisfying $\Delta^{2}v=-v^{-q}$ in $\mathbb{R}^{n-1}$. This leaves the question whether in $\mathbb{R}^{3}$ non-radial positive entire solution can be constructed in such way or not. Theorem \ref{thm0} gives a negative answer to the open question (2) raised by Mckenna and Reichel in Section 6 of \cite{MR} and hence non-radial positive entire solution in $\mathbb{R}^{3}$ can not be constructed from radial positive entire solution in $\mathbb{R}^{2}$.
\end{rem}

\begin{rem}\label{rem10}
After this work has been completed and submitted, we were aware that the nonexistence results in Theorem \ref{thm0} for the 2D bi-harmonic equation \eqref{PDE} has already been proved by Ng\^{o}, Nguyen, Phan and Ye in Proposition 4.1 of \cite{NNPY}. Their proof makes use of the results on super/sub poly-harmonic properties in Lemma 3.3 of \cite{NNPY} and the Liouville type theorem in $\mathbb{R}^{2}$ for super-harmonic functions (bounded from below) in Theorem 3.1 of Farina \cite{Farina}. Theorem 3.1 in \cite{Farina} is proved by using the Hadamard three-circles theorem (Theorem 3.2 in \cite{Farina}). One should notice that Theorem 3.1 in \cite{Farina} only holds in $\mathbb{R}^{2}$, hence $\mathbb{R}^{2}$ (endowed with the standard flat metric) is a parabolic Riemannian manifold. Euclidean space $\mathbb{R}^{n}$ with $n\gee 3$ (endowed with the standard flat metric) is not a parabolic Riemannian manifold. Indeed, for any $n\gee 3$, the non-constant positive function $u(x):=\Big(\frac{\sqrt{n(n-2)}}{1+|x|^{2}}\Big)^{\frac{n-2}{2}}$ solves the Yamabe equation $-\Delta u=u^{\frac{n+2}{n-2}}$ in $\mathbb{R}^{n}$. Thus the Liouville type theorem for super-harmonic functions (bounded from below) in Theorem 3.1 of \cite{Farina} does not hold in $\mathbb{R}^{n}$ with $n\gee 3$. In Section 3, we reprove the nonexistence results in Proposition 4.1 of \cite{NNPY} for the 2D bi-harmonic equation \eqref{PDE} and gives another completely different approach to Theorem \ref{thm0} which is interesting and could be instructive for studying other related problems. Our proof first makes use of the spherical averages (see Lemma \ref{lem2}), the 2D comparison theorem (see Theorem \ref{thm2}), precise asymptotic estimates and integral representation formula (see Theorems \ref{thm3} and \ref{thm4}, Lemmas \ref{lem4} and \ref{lem3}) to prove the nonexistence of positive entire solutions $u$ with $u^{-1}$ bounded from above or radial symmetry, then apply the doubling lemma (see Lemma \ref{doubling}) to derive the nonexistence of general positive entire solutions.
\end{rem}

The rest of our paper are arranged as follows. Section 2 is devoted to the proof of various properties (including nonexistence, asymptotic behavior, uniqueness and integral representation formula) of positive solutions to the tri-harmonic equation \eqref{PDE}, i.e.,  Theorem \ref{thm1}. In Section 3, we will prove the nonexistence of positive solutions to the 2D bi-harmonic equation \eqref{PDE}, i.e.,  Theorem \ref{thm0}.

\smallskip

In what follows, we will use $C$ to denote a general positive constant that may depend on $n$ and $q$, and whose value may differ from line to line.

\section{Tri-harmonic equations in $\mathbb{R}^{n}$ with $n\gee 2$}

In this section, we will prove various properties (including nonexistence, asymptotic behavior, uniqueness and integral representation formula) of positive solutions to the tri-harmonic equation \eqref{PDE}, i.e.,  Theorem \ref{thm1}.

\begin{lem}\label{lem0}
Assume $n\gee1$ and $q>0$. For any point $x_{0}$ $\in \mathbb{R}^{n}$ and all $r>0$,
\begin{equation}\label{2}
\Bigg[\ \ \ -\kern-22.5pt\int\limits_{\p B_{r}(x_{0})} u\,d\sigma \Bigg]^{-q}\ \lee \ \ \ -\kern-22.5pt\int\limits_{\p B_{r}(x_{0})} u^{-q}\,d\sigma,
\end{equation}
where the symbol $-\kern-10pt\int\limits_{S} f\,d\sigma$ denotes the spherical average of function $f$ over the sphere $S$.
\end{lem}
\begin{proof}
Due to the convexity of the function $f(u)=u^{-q}$ on the interval $(0,+\infty)$, Lemma \ref{lem0} follows immediately from Jensen's inequality.
\end{proof}

From Theorem 1 in \cite{N}, we can derive the following sub poly-harmonic property for $C^{6}$ positive solution $u$ to the tri-harmonic equation \eqref{PDE}.
\begin{lem}\label{lem5}
Assume $m=3$, $n\gee 2$ and $q>0$. If $u\in C^{6}(\mathbb{R}^{n})$ is a positive entire solution in $\mathbb{R}^{n}$ to \eqref{PDE} satisfying $u(x)=o({|x|}^{4})$ when $|x|\rt+\infty$, then $\Delta u>0$ and ${\Delta}^{2} u<0$ in $\mathbb{R}^{n}$.
\end{lem}

Now we define $w:=\Delta u$, $v:=\Delta w=\Delta^{2}u$ and
\begin{equation}\label{3-3}
\bar{u}(r):=\ \ \  -\kern-20.5pt\int\limits_{\p B_{r}(0)} u\,d\sigma, \qquad \bar{w}(r):=\ \ \  -\kern-20.5pt\int\limits_{\p B_{r}(0)}\Delta u\,d\sigma, \qquad \bar{v}(r)=\ \ \  -\kern-20.5pt\int\limits_{\p B_{r}(0)}\Delta w\,d\sigma, \qquad \forall \, r\gee 0.
\end{equation}
Recall that, in $\mathbb{R}^{n}$, for any radially symmetric function $f(r)$, $\Delta f(r)=\frac{1}{{r}^{n-1}}\left({r}^{n-1}f'(r)\right)'$. It can be deduced from Lemma \ref{lem0} that $\bar{u}(r)$, $\bar{w}(r)$ and $\bar{v}(r)$ satisfy
\begin{align}\label{6-3}
\left\{
\begin{aligned}
&\Delta\bar{u}(r)=\bar{w}(r), \quad \forall \, r\gee 0, \\
&\Delta\bar{w}(r)=\bar{v}(r), \quad \forall \, r\gee 0, \\
&\Delta\bar{v}(r)-{\bar{u}}^{-q}(r)\gee 0, \quad \forall \, r\gee 0.
\end{aligned}
\right.
\end{align}

We have the following lemma.
\begin{lem}\label{lem6}
Assume $m=3$, $n\gee 2$ and $q>0$. If $u\in C^{6}(\mathbb{R}^{n})$ is a positive entire solution in $\mathbb{R}^{n}$ to \eqref{PDE}, then for all $r>0$,
\begin{equation}\label{7-3}
\quad \bar{v}'(r)>0.
\end{equation}
Moreover, if $u$ satisfies the sub poly-harmonic property ``$\Delta^{2}u<0$ and $\Delta u>0$ in $\mathbb{R}^{n}$", then for all $r>0$,
\begin{equation}\label{7-3'}
  \bar{u}'(r)>0, \quad \bar{w}'(r)<0.
\end{equation}
\end{lem}
\begin{proof}
Multiplying the third inequality in \eqref{6-3} by $r^{n-1}$ and integrating the resulting equation, we get
\begin{equation}\label{a1}
{r}^{n-1}\bar{v}'(r)-\int_{0}^{r} {t}^{n-1}{\bar{u}}^{-q}(t)dt\gee 0.
\end{equation}
Then the inequality in \eqref{7-3} follows immediately. From the second inequality in \eqref{6-3} and the sub poly-harmonic property, we get $\left({r}^{n-1}\bar{w}'(r)\right)'={r}^{n-1}\bar{v}(r)<0$ for any $r>0$, and hence the second inequality in \eqref{7-3'} follows immediately by integrating. Similarly, from the first inequality in \eqref{6-3} and the sub poly-harmonic property, we get $({r}^{n-1}\bar{u}'(r))'={r}^{n-1}\bar{w}(r)>0$ for any $r>0$, and hence the first inequality in \eqref{7-3'} follows immediately by integrating. This finishes our proof of Lemma \ref{lem6}.
\end{proof}
\begin{rem}\label{rem4}
By Lemma \ref{lem5}, the sub poly-harmonic property ``$\Delta^{2}u<0$ and $\Delta u>0$ in $\mathbb{R}^{n}$" can be deduced from the assumption $u=o(|x|^{4})$ at $\infty$. Therefore, for any positive entire solution $u$ to \eqref{PDE} satisfying $u=o(|x|^{4})$ at $\infty$, we have $\bar{u}'(r)>0$ and $\bar{w}'(r)<0$ for any $r>0$.
\end{rem}
\begin{rem}\label{rem3}
In fact, we can also show that the property ``$\bar{w}''(r)<0$ and $\bar{w}'''(r)>0$ for any $r>0$" does not hold in general cases where the sub poly-harmonic property ``$\Delta^{2}u<0$ and $\Delta u>0$ in $\mathbb{R}^{n}$" holds. By direct calculations, we obtain that
\begin{equation}\label{14-3}
{\Delta}^{3}\bar{u}(r)={\Delta}^{2}\bar{w}(r)=\frac{1}{r^{n+1}}\left(r^{n+1}{\bar{w}}^{(3)}(r)\right)'+\frac{n-3}{r}\bar{v}'(r), \qquad \forall \, r>0.
\end{equation}
If $\lim\limits_{r\rightarrow+\infty}\bar{w}''(r)=\beta$ and $\lim\limits_{r\rightarrow+\infty}r{\bar{w}}^{(3)}(r)=\gamma$ exist, then by \eqref{PDE}, \eqref{14-3} and integrating, we can get
\begin{eqnarray}\label{a2}
  && \quad \bar{w}(r)-\bar{w}(0) \\
 \nonumber &&=\frac{\beta}{2}r^{2}+\frac{\gamma}{2n}r^{2}-\frac{1}{n(n-1)(n-2)r^{n-2}}
  \int_{0}^{r}t^{n+1}\left(\overline{u^{-q}}(t)-\frac{n-3}{t}\bar{v}'(t)\right)dt \\
 \nonumber && \quad -\frac{r^{2}}{2n}\int_{r}^{+\infty}t\left(\overline{u^{-q}}(t)-\frac{n-3}{t}\bar{v}'(t)\right)dt
  +\frac{1}{2(n-2)}\int_{0}^{r}t^{3}\left(\overline{u^{-q}}(t)-\frac{n-3}{t}\bar{v}'(t)\right)dt \\
 \nonumber && \quad -\frac{r}{n-1}\int_{0}^{r}t^{2}\left(\overline{u^{-q}}(t)-\frac{n-3}{t}\bar{v}'(t)\right)dt\\
 \nonumber && =: \frac{\beta}{2}r^{2}+\frac{\gamma}{2n}r^{2}+\Phi(r), \qquad \forall \, r>0.
\end{eqnarray}

\smallskip

Assume that $n\gee 3$, $q>3$ and $u(x)=u(r)$ with $r=|x|$ is radially symmetric and satisfies $\liminf\limits_{|x|\rightarrow+\infty}\frac{u(x)}{|x|}\in(0,+\infty]$ (when $n=5$, this condition can be deduced from $u=o(|x|^{4})$ at $\infty$, see Lemma \ref{lem7}; for general $n\gee 2$, this condition can be deduced from $q\lee 4$, see Theorem \ref{thm6}), then by L'Hopital's rule, one has
\begin{eqnarray}\label{a3}
  && \gamma=\lim_{r\rightarrow+\infty}r{\bar{w}}^{(3)}(r)
  =\lim_{r\rightarrow+\infty}\frac{\int_{0}^{r}t^{n+1}\left(\overline{u^{-q}}(t)-\frac{n-3}{t}\bar{v}'(t)\right)dt}{r^{n}} \\
  \nonumber && \quad =\lim_{r\rightarrow+\infty}\frac{r^{2}}{n}\left(\overline{u^{-q}}(r)-\frac{n-3}{r}\bar{v}'(r)\right)
  =\frac{1}{n(n-2)}\lim_{r\rightarrow+\infty}r^{2}\overline{u^{-q}}(r)=0,
\end{eqnarray}
and
\begin{eqnarray}\label{a4}
  && \lim_{r\rightarrow+\infty}\Phi'(r)=\lim_{r\rightarrow+\infty}\bigg[\frac{1}{n(n-1)r^{n-1}}
  \int_{0}^{r}t^{n+1}\left(\overline{u^{-q}}(t)-\frac{n-3}{t}\bar{v}'(t)\right)dt \\
  \nonumber && \quad -\frac{r}{n}\int_{r}^{+\infty}t\left(\overline{u^{-q}}(t)-\frac{n-3}{t}\bar{v}'(t)\right)dt
  -\frac{1}{n-1}\int_{0}^{r}t^{2}\left(\overline{u^{-q}}(t)-\frac{n-3}{t}\bar{v}'(t)\right)dt\bigg] \\
  \nonumber &&=\frac{2-n}{(n-1)^{2}}\lim_{r\rightarrow+\infty}r^{3}\left(\overline{u^{-q}}(r)-\frac{n-3}{r}\bar{v}'(r)\right)
  -\frac{1}{n-1}\int_{0}^{+\infty}t^{2}\left(\overline{u^{-q}}(t)-\frac{n-3}{t}\bar{v}'(t)\right)dt \\
  \nonumber && =-\frac{1}{n-1}\int_{0}^{+\infty}t^{2}\left(\overline{u^{-q}}(t)-\frac{n-3}{t}\bar{v}'(t)\right)dt.
\end{eqnarray}

\smallskip

Now suppose that ``$\bar{w}''(r)<0$ and $\bar{w}'''(r)>0$ for any $r>0$" holds for such radially symmetric positive entire solution $u$, then $\lim\limits_{r\rightarrow+\infty}\bar{w}''(r)=\beta\lee 0$ exists. Therefore, it follows from \eqref{a2}, \eqref{a3} and \eqref{a4} that
\begin{equation}\label{a4}
  0>\bar{w}(r)-\bar{w}(0)=\frac{\beta}{2}r^{2}-\left[\frac{1}{n-1}\int_{0}^{+\infty}t^{2}\left(\overline{u^{-q}}(t)-\frac{n-3}{t}\bar{v}'(t)\right)dt\right]r+o(r),
\end{equation}
as $r\rightarrow+\infty$. Hence $\bar{w}(r)\rightarrow-\infty$ as $r\rightarrow+\infty$, which contradicts the sub poly-harmonic property ``$\Delta^{2}u<0$ and $\Delta u>0$ in $\mathbb{R}^{n}$" or Lemma \ref{lem5} provided that $u=o(|x|^{4})$ at $\infty$.
\end{rem}

\medskip

By Lemma \ref{lem6}, if the positive entire solution $u$ satisfies the sub poly-harmonic property ``$\Delta^{2}u<0$ and $\Delta u>0$ in $\mathbb{R}^{n}$" or $u=o(|x|^{4})$ at $\infty$, then $\bar{w}'(r)<0$ for any $r>0$, and hence $\bar{w}(r)=\Delta\bar{u}(r)\lee \bar{w}(0)=w(0)=\Delta u(0)$. Now by integrating again, we arrive at
\begin{equation}\label{7-3a}
\bar{u}(r)\lee \bar{u}(0)+\frac{\bar{w}(0)}{4}r^{2}=u(0)+\frac{\Delta u(0)}{4}r^{2}, \qquad \forall \, r\gee 0.
\end{equation}
Consequently, any $C^{6}$ positive entire solution $u$ to the tri-harmonic equation \eqref{PDE} must satisfy
\begin{equation}\label{e11-3}
  \liminf_{|x|\rightarrow+\infty}\frac{u(x)}{|x|^{2}}\lee\frac{\Delta u(0)}{4}<+\infty.
\end{equation}

\begin{rem}\label{rem5}
In this section, we only need the assumption $u=o(|x|^{4})$ at $\infty$ to guarantee the sub poly-harmonic property ``$\Delta^{2}u<0$ and $\Delta u>0$ in $\mathbb{R}^{n}$" (see Lemma \ref{lem5}). We may replace the condition ``$u=o(|x|^{4})$ at $\infty$" by the sub poly-harmonic property ``$\Delta^{2}u<0$ and $\Delta u>0$ in $\mathbb{R}^{n}$" everywhere hereafter in this section. The upper bound estimate \eqref{7-3a} indicates that the assumption $u=o(|x|^{4})$ at $\infty$ is necessary in the sense that, if the sub poly-harmonic property ``$\Delta^{2}u<0$ and $\Delta u>0$ in $\mathbb{R}^{n}$" holds, then $u$ has no more than quadratic growth at $\infty$ in the sense of spherical average, i.e., $\bar{u}(r)\lee Cr^{2}$ for $r$ sufficiently large.
\end{rem}

We can deduce the following lemma on asymptotic behaviors of $u$, $w$ and $v$ as $|x|\rightarrow+\infty$.
\begin{lem}\label{lem7}
Assume $m=3$, $n\gee 2$ and $q>0$. If $u\in C^{6}(\mathbb{R}^{n})$ is a positive entire solution to the tri-harmonic equation \eqref{PDE} satisfying $u(x)=o({|x|}^{4})$ as $|x|\rt+\infty$, then
\begin{equation}\label{a5}
  \bar{v}(r)\gee -\frac{2n\Delta u(0)}{r^{2}}, \qquad \forall \, r>0;
\end{equation}
\begin{equation}\label{a6}
  v(x)\lee -\frac{c}{|x|^{n-2}}, \quad \forall \,|x|\gee 1 \qquad \text{if} \,\, n\gee 3,
\end{equation}
\begin{equation}\label{a6'}
  v(x)\lee -\frac{c_{k}}{\ln^{(k)}(|x|)}, \quad \forall \,|x|\gee {\exp}^{(k)}(1), \quad \forall \, k\in\mathbb{N}^{+} \qquad \text{if} \,\, n=2,
\end{equation}
where $c=-\max\limits_{|x|=1}v(x)>0$ if $n\gee 3$, $c_{k}=-\max\limits_{|x|={\exp}^{(k)}(1)}v(x)>0$ for every $k\gee 1$ if $n=2$, $\ln^{(k)}:=\underbrace{\ln\cdots\ln}_{k \, \text{times}}$ and ${\exp}^{(k)}:=\underbrace{\exp\cdots\exp}_{k \, \text{times}}$;
\begin{equation}\label{a7}
  w(x)\gee \frac{c}{|x|^{n-4}}, \quad \forall \,|x|\gee 1 \qquad \text{if} \,\, n\gee 5,
\end{equation}
where $c=\min\left\{\frac{-1}{2(n-4)}\max\limits_{|x|=1}v(x),\min\limits_{|x|=1}w(x)\right\}>0$;
\begin{equation}\label{a8}
  \bar{u}(r)\gee \frac{c}{10}r, \quad \forall \, r\gee 1 \qquad \text{if} \,\, n=5,
\end{equation}
where the constant $c$ is the same as in \eqref{a7}.
\end{lem}
\begin{proof}
By the second equation in \eqref{6-3} and Lemma \ref{lem6}, we get
\begin{equation}
r^{n-1}\bar{w}'(r)=\int_{0}^{r}t^{n-1}\bar{v}(t)dt \lee \frac{1}{n}r^{n}\bar{v}(r), \qquad \forall \, r\gee 0.
\end{equation}
Divide this equation by $r^{n-1}$ and integrate once again to get
\begin{equation}\label{16-3}
\bar{w}(r)\lee \bar{w}(0)+\frac{1}{2n}r^{2}\bar{v}(r), \qquad \forall \, r\gee 0.
\end{equation}
It follows that
\begin{equation}\label{17-3}
-\bar{v}(r)\lee \frac{2n\left(\bar{w}(0)-\bar{w}(r)\right)}{r^{2}}\lee \frac{2n\bar{w}(0)}{r^{2}}=\frac{2n\Delta u(0)}{r^{2}}, \qquad \forall \, r>0.
\end{equation}

If $n\gee 3$, set $c=-\max\limits_{|x|=1}v(x)$. By Lemma \ref{lem5}, one has $v=\Delta^{2} u<0$ and hence $c>0$. Applying the maximum principle to the function $\frac{c}{|x|^{n-2}}+v$ on the region $\{x\in \mathbb{R}^{n}| \, 1\lee |x| <+\infty\}$, we obtain that
\begin{equation}\label{e33-3}
  v(x)\lee -\frac{c}{|x|^{n-2}}, \qquad \forall \, |x| \gee 1.
\end{equation}
If $n=2$, for arbitrary $k\in\mathbb{N}^{+}$, set $c_{k}=-\max\limits_{|x|={\exp}^{(k)}(1)}v(x)$, where $\exp^{(k)}:=\underbrace{\exp\cdots\exp}_{k \, \text{times}}$. By Lemma \ref{lem5}, one has $v=\Delta^{2} u<0$ and hence $c_{k}>0$. Applying the maximum principle to the function $\frac{c_{k}}{{ln}^{(k)}(|x|)}+v$ on the region $\{x\in \mathbb{R}^{2}| \, {\exp}^{(k)}(1) \lee |x| <+\infty\}$, we obtain that
\begin{equation}\label{e33-3'}
 v(x)\lee -\frac{c_{k}}{{ln}^{(k)}(|x|)}, \qquad \forall \, |x| \gee {\exp}^{(k)}(1), \quad \forall \, k\gee 1.
\end{equation}

If $n\gee 5$, set $c=\min\left\{\frac{-1}{2(n-4)}\max\limits_{|x|=1}v(x),\min\limits_{|x|=1}w(x)\right\}$. By Lemma \ref{lem5}, one has $v=\Delta^{2} u<0$, $w=\Delta u>0$ and hence $c>0$. By \eqref{a6}, we have $\Delta\left(w-\frac{c}{|x|^{n-4}}\right)=v+\frac{2(n-4)c}{|x|^{n-2}}\lee 0$ for any $|x|\gee 1$. Applying the maximum principle to the function $w-\frac{c}{|x|^{n-4}}$ on the region $\{x\in \mathbb{R}^{n}| \, 1 \lee |x| <+\infty\}$, we obtain that
\begin{equation}\label{e33-3''}
 w(x)\gee \frac{c}{|x|^{n-4}}, \qquad \forall \, |x| \gee 1.
\end{equation}

Using the first equation in \eqref{6-3} and Lemma \ref{lem6}, we get
\begin{equation}\label{e21-3}
r^{n-1}\bar{u}'(r)=\int_{0}^{r}t^{n-1}\bar{w}(t)dt \gee \frac{1}{n}r^{n}\bar{w}(r), \qquad \forall \, r\gee 0.
\end{equation}
By dividing \eqref{e21-3} by $r^{n-1}$ and integrating once again, we get
\begin{equation}\label{13-3}
\bar{u}(r) \gee u(0) + \frac{1}{2n}r^{2}\bar{w}(r), \qquad \forall \, r\gee 0.
\end{equation}
If $n=5$, the lower bound estimate \eqref{a8} follows directly from \eqref{a7} and \eqref{13-3}. This completes our proof of Lemma \ref{lem7}.
\end{proof}

The lower bound estimate \eqref{a8} can be improved remarkably and the assumption $n=5$ can also be removed. As a consequence, we can derive the necessary condition for the existence of positive entire solutions to the tri-harmonic equation \eqref{PDE}.
\begin{thm}\label{thm6}
Assume $m=3$, $n\gee 2$ and $q>0$. Suppose $u\in C^{6}(\mathbb{R}^{n})$ is a positive entire solution to the tri-harmonic equation \eqref{PDE} satisfying $u=o(|x|^{4})$ as $|x|\rightarrow+\infty$, then
\begin{equation}\label{a9}
 \bar{u}(r)\gee Cr^{\frac{4}{q}}, \qquad \forall \, r\gee 0,
\end{equation}
and
\begin{equation}\label{a11}
  \int_{B_{r}(0)}u^{-q}(x)dx\lee Cr^{n-4}, \qquad \forall \, r>0.
\end{equation}
Moreover, if $n=4$, then
\begin{equation}\label{a12}
  \int_{\mathbb{R}^{4}}u^{-q}(x)dx<+\infty.
\end{equation}
Consequently, if the tri-harmonic equation \eqref{PDE} admits a positive entire solution $u\in C^{6}(\mathbb{R}^{n})$ such that $u(x)=o({|x|}^{4})$ as $|x|\rt+\infty$, then $q\gee 2$ and $n\gee 4$.
\end{thm}
\begin{proof}
Suppose $u\in C^{6}(\mathbb{R}^{n})$ is a positive entire solution to the tri-harmonic equation \eqref{PDE} satisfying $u=o(|x|^{4})$ at $\infty$. By \eqref{PDE}, Lemma \ref{lem0} and Lemma \ref{lem6}, we have
\begin{align}\label{18-3}
&\quad \bar{v}(2r)-\bar{v}(r)=\int_{r}^{2r}\bar{v}'(t)dt \\
& = \frac{1}{\Sigma_{n-1}}\int_{r}^{2r}t^{-(n-1)}\int_{B_{t}(0)}\Delta{v}dxdt \notag \\
& = \frac{1}{\Sigma_{n-1}}\int_{r}^{2r}t^{-(n-1)}\int_{B_{t}(0)}u^{-q}dxdt \notag \\
& \gee \frac{C_{n}}{\Sigma_{n-1}r^{n-2}}\int_{B_{r}(0)}u^{-q}(x)dx \notag \\
& = \frac{C_{n}}{r^{n-2}}\int_{0}^{r}t^{n-1} -\kern-13pt\int_{\partial B_{t}(0)}u^{-q}(x)d{\sigma}dt \notag \\
& \gee \frac{C_{n}}{r^{n-2}}\int_{0}^{r}t^{n-1}{\bar{u}}^{-q}(t)dt \notag \\
& \gee C_{n}r^{2}{\bar{u}}^{-q}\left(r\right), \qquad \forall \, r\gee 0, \notag
\end{align}
where $\Sigma_{n-1}$ denotes the surface area of the unit $(n-1)$-sphere in $\mathbb{R}^{n}$. By Lemma \ref{lem5}, we can deduce from \eqref{a5} in Lemma \ref{lem7} and \eqref{18-3} that
\begin{equation}\label{a10}
  C_{n}r^{2}{\bar{u}}^{-q}\left(r\right)\lee \bar{v}(2r)-\bar{v}(r)<\frac{2n\Delta u(0)}{r^{2}}, \qquad \forall \, r>0,
\end{equation}
and hence there is a positive constant $C$ such that
\begin{equation}\label{19-3}
\bar{u}(r)\gee Cr^{\frac{4}{q}}, \qquad \forall \, r\gee 0,
\end{equation}
that is, the lower bound estimate \eqref{a9} holds.

Now suppose the tri-harmonic equation \eqref{PDE} admits a positive entire solution $u\in C^{6}(\mathbb{R}^{n})$ such that $u(x)=o({|x|}^{4})$ as $|x|\rt \infty$. By the lower bound estimate \eqref{a9} and the upper bound estimate \eqref{7-3a}, we must have $\frac{4}{q}\lee 2$, that is, $q\gee 2$. From the first inequality in \eqref{18-3} and \eqref{a10}, we can also infer that
\begin{equation}\label{a11'}
  \int_{B_{r}(0)}u^{-q}(x)dx\lee Cr^{n-4}, \qquad \forall \, r>0,
\end{equation}
which will yield a contradiction if $n=2,3$. If $n=2,3$, we can also derive a contradiction from \eqref{a5}, \eqref{a6} and \eqref{a6'} in Lemma \ref{lem7}. This completes our proof of Theorem \ref{thm6}.
\end{proof}

From Theorem \ref{thm6}, we know that \eqref{PDE} admits no positive entire solution $u\in C^{6}(\mathbb{R}^{n})$ satisfying ``$u=o(|x|^{4})$ at $\infty$" or sub poly-harmonic property provided that $0<q<2$ or $n=2,3$. Therefore, in the rest of this section, we only need to consider the cases $q\gee 2$ and $n\gee 4$ when discussing the properties of $C^{6}$ positive entire solution $u$ to \eqref{PDE} satisfying ``$u=o(|x|^{4})$ at $\infty$" or sub poly-harmonic property.

\medskip

We can deduce from Lemma \ref{lem7} and Theorem \ref{thm6} the following Corollary immediately.
\begin{cor}\label{cor0}
Assume $m=3$, $n\gee 4$ and $q\gee 2$. Then, we have \\
(i) \, Equation \eqref{PDE} admits no positive entire solution $u\in C^{6}(\mathbb{R}^{n})$ satisfying $u=o\left(|x|^{\frac{4}{q}}\right)$ as $|x|\rightarrow+\infty$. Moreover, \eqref{PDE} admits no positive radially symmetric entire solution $u\in C^{6}(\mathbb{R}^{n})$ satisfying $u=o(r^{4})$ at $\infty$ and
\begin{equation}\label{e8-3a}
  \liminf_{r\rightarrow+\infty}\frac{u(r)}{r^{\frac{4}{q}}}=0.
\end{equation}
(ii) \, For $n=5$, suppose $u$ is a positive radially symmetric $C^{6}$ entire solution to \eqref{PDE} satisfying $u=o(|x|^{4})$ at $\infty$, then $u(r)\gee cr$ for some positive constant $c>0$ and any $r:=|x|\gee 0$. That is, $\liminf\limits_{|x|\rightarrow+\infty}\frac{u(x)}{|x|}\in(0,+\infty]$. \\
(iii) \, If $u$ has at most linear growth (uniformly) at $\infty$, that is,
\begin{equation}\label{e9a}
  \lim_{|x|\rightarrow+\infty}\frac{u(x)}{|x|}=\alpha\in[0,+\infty),
\end{equation}
then $q\gee 4$. In addition, we have $\alpha>0$ if $n=5$. \\
(iv) \, Suppose $u\in C^{6}(\mathbb{R}^{n})$ is a positive entire solution to \eqref{PDE} satisfying $u=o(|x|^{4})$ at $\infty$, then there exists a constant $C>0$ such that
\begin{equation}\label{d5}
  \bar{v}'(r)\lee \frac{C}{r^{3}}, \qquad \forall \, r>0,
\end{equation}
and
\begin{equation}\label{d0}
  \liminf_{r\rightarrow+\infty}r^{4}\overline{u^{-q}}(r)<+\infty.
\end{equation}
\end{cor}
\begin{proof}
Conclusions (i)-(iii) can be deduced from Lemma \ref{lem7} and Theorem \ref{thm6} immediately, we omit the details. We only show (iv). From \eqref{a11} in Theorem \ref{thm6}, we obtain
\begin{equation}\label{d6}
  \bar{v}'(r)=\frac{1}{r^{n-1}}\int_{0}^{r}t^{n-1}\overline{u^{-q}}(t)dt=\frac{\int_{B_{r}(0)}u^{-q}(x)dx}{\Sigma_{n-1}r^{n-1}}\lee \frac{Cr^{n-4}}{\Sigma_{n-1}r^{n-1}}=:\frac{C}{r^{3}}, \qquad \forall \,r>0.
\end{equation}
Estimate \eqref{a11} in Theorem \ref{thm6} also implies that, for any $r>0$,
\begin{equation}\label{d1}
  c_{n}\Sigma_{n-1}r^{n-4}\min_{t\in[r,2r]}t^{4}\overline{u^{-q}}(t)\lee \int_{r}^{2r}\Sigma_{n-1}t^{n-1}\overline{u^{-q}}(t)dt\lee \int_{B_{2r}(0)}u^{-q}(x)dx\lee Cr^{n-4},
\end{equation}
where $c_{n}:=\frac{2^{n-4}-1}{n-4}$ for $n\neq4$ and $c_{4}:=\ln 2$ for $n=4$. That means,
\begin{equation}\label{d3}
  \min_{t\in[r,2r]}t^{4}\overline{u^{-q}}(t)\lee C, \qquad \forall \, r\gee 0.
\end{equation}
Suppose on the contrary that \eqref{d0} does not hold, then
\begin{equation}\label{d2}
  \lim_{r\rightarrow+\infty}r^{4}\overline{u^{-q}}(r)=+\infty,
\end{equation}
which contradicts \eqref{d3} if $r$ is sufficiently large. This finishes our proof of Corollary \ref{cor0}.
\end{proof}

\begin{thm}\label{thm7}
Assume $n=4$, $m=3$ and $q\gee 2$. Equation \eqref{PDE} admits no positive solutions $u\in C^{6}(\mathbb{R}^{4})$ on entire $\mathbb{R}^{4}$ satisfying $u=o(|x|^{4})$ as $|x|\rightarrow+\infty$.
\end{thm}
\begin{proof}
Suppose on the contrary that $u\in C^{6}(\mathbb{R}^{4})$ is a positive entire solution to the tri-harmonic equation \eqref{PDE} satisfying $u=o(|x|^{4})$ as $|x|\rightarrow+\infty$. We will obtain a contradiction.

The lower bound estimate \eqref{a5} and the upper bound estimate \eqref{a6} in Lemma \ref{lem7} imply that for all $r\gee 1$,
\begin{equation}\label{a13}
  -\frac{8\Delta u(0)}{r^{2}}\lee \bar{v}(r)\lee -\frac{c}{r^{2}},
\end{equation}
where $c=-\max\limits_{|x|=1}v(x)>0$, and hence
\begin{equation}\label{42-3}
-8\Delta u(0)r\lee (r^{3}\bar{w}')'(r)\lee -cr, \qquad \forall \, r\gee 1.
\end{equation}
Integrating from $\frac{r}{2}$ to $r$, we get
\begin{equation}\label{43-3}
-\frac{3}{4}cr^{2}\gee r^{3}\bar{w}'(r)-\left(\frac{r}{2}\right)^{3}\bar{w}'\left(\frac{r}{2}\right)\gee -3\Delta u(0)r^{2}, \qquad \forall
\,r\gee 2.
\end{equation}
By Lemma \ref{lem6}, it follows that
\begin{equation}\label{44-3}
 r^{3}\bar{w}'(r)\lee -\frac{3}{4}cr^{2}, \qquad \forall \,r\gee 2.
\end{equation}
Divide by $r^{3}$ and Integrate from $2$ to $r$, we get
\begin{equation}\label{44-3}
\bar{w}(r)\lee -\frac{3}{4}c\left[\ln r-\ln 2\right]+\bar{w}(2),
\end{equation}
which implies that $\bar{w}(r)\rightarrow-\infty$ as $r\rightarrow+\infty$. This contradicts the fact $w=\Delta u>0$ in Lemma \ref{lem5}. This finishes our proof of Theorem \ref{thm7}.
\end{proof}

\begin{thm}\label{thm8}
Assume $m=3$, $n\gee2$ and $q>0$, equation \eqref{PDE} admits no positive entire solution $u\in C^{6}(\mathbb{R}^{n})$ satisfying $u=o\left(|x|^{\frac{6}{q+1}}\right)$ as $|x|\rightarrow+\infty$. Moreover, \eqref{PDE} admits no positive radially symmetric entire solution $u\in C^{6}(\mathbb{R}^{n})$ satisfying
\begin{equation}\label{a28}
  \limsup_{r\rightarrow+\infty}\frac{u(r)}{r^{\frac{6}{q+1}}}<+\infty \qquad \text{and} \qquad \liminf_{r\rightarrow+\infty}\frac{u(r)}{r^{\frac{6}{q+1}}}=0.
\end{equation}
\end{thm}
\begin{proof}
Suppose on the contrary that $u\in C^{6}(\mathbb{R}^{n})$ is a positive entire solution to \eqref{PDE} satisfying $u=o\left(|x|^{\frac{6}{q+1}}\right)$ at $\infty$. Since $u=o\left(|x|^{\frac{6}{q+1}}\right)$ as $|x|\rightarrow+\infty$, we have
\begin{equation}\label{a16}
  \bar{u}(r)\lee r^{\frac{6}{q+1}}, \qquad \text{for} \,\, r \,\, \text{large enough}.
\end{equation}
By Lemma \ref{lem0}, we infer from \eqref{a16} that
\begin{equation}\label{a17}
  \Delta^{3}\bar{u}(r)=\frac{1}{r^{n-1}}\left(r^{n-1}\bar{v}'(r)\right)'(r)\gee \bar{u}^{-q}(r)\gee r^{-\frac{6q}{q+1}},  \qquad \text{for} \,\, r \,\, \text{large enough}.
\end{equation}
By Lemma \ref{lem6} and integrating from $r$ to $2r$, we get
\begin{equation}\label{a18}
  (2r)^{n-1}\bar{v}'(2r)\gee (2r)^{n-1}\bar{v}'(2r)-r^{n-1}\bar{v}'(r)\gee \int_{r}^{2r}s^{n-1-\frac{6q}{q+1}}ds=Cr^{n-\frac{6q}{q+1}}
\end{equation}
for $r$ sufficiently large, and hence
\begin{equation}\label{a19}
   -\bar{v}(r)\gee \bar{v}(2r)-\bar{v}(r)\gee C\int_{r}^{2r}s^{1-\frac{6q}{q+1}}ds=Cr^{2-\frac{6q}{q+1}} \qquad \text{for} \,\, r \,\, \text{large enough}.
\end{equation}
By Lemma \ref{lem6} and integrating from $r$ to $2r$ again, we obtain
\begin{equation}\label{a20}
  (2r)^{n-1}\bar{w}'(2r)\lee (2r)^{n-1}\bar{w}'(2r)-r^{n-1}\bar{w}'(r)\lee -C\int_{r}^{2r}s^{n+1-\frac{6q}{q+1}}ds=-Cr^{n+2-\frac{6q}{q+1}}
\end{equation}
for $r$ sufficiently large, and hence
\begin{equation}\label{a21}
   -\bar{w}(r)\lee \bar{w}(2r)-\bar{w}(r)\lee -C\int_{r}^{2r}s^{3-\frac{6q}{q+1}}ds=-Cr^{4-\frac{6q}{q+1}} \qquad \text{for} \,\, r \,\, \text{large enough}.
\end{equation}
By Lemma \ref{lem6} and integrating from $r$ to $2r$ again, we arrive at
\begin{equation}\label{a22}
  (2r)^{n-1}\bar{u}'(2r)\gee (2r)^{n-1}\bar{u}'(2r)-r^{n-1}\bar{u}'(r)\gee C\int_{r}^{2r}s^{n+3-\frac{6q}{q+1}}ds=Cr^{n+4-\frac{6q}{q+1}}
\end{equation}
for $r$ sufficiently large, and hence
\begin{equation}\label{a23}
   \bar{u}(2r)\gee \bar{u}(2r)-\bar{u}(r)\gee C\int_{r}^{2r}s^{5-\frac{6q}{q+1}}ds=Cr^{6-\frac{6q}{q+1}} \qquad \text{for} \,\, r \,\, \text{large enough}.
\end{equation}
By \eqref{a23}, we must have $6-\frac{6q}{q+1}=\frac{6}{q+1}<\frac{6}{q+1}$ as $u=o\left(|x|^{\frac{6}{q+1}}\right)$ at $\infty$, which is absurd. This finishes our proof of Theorem \ref{thm8}.
\end{proof}
\begin{rem}\label{rem6}
By Theorem \ref{thm8}, the conclusions in (iii) of Corollary \ref{cor0} can be improved to the following: \\
$(iii')$ \, For any $n\gee 2$ and $q>0$, if $u$ has at most linear growth (uniformly) at $\infty$, that is,
\begin{equation}\label{e9-3a}
  \lim_{|x|\rightarrow+\infty}\frac{u(x)}{|x|}=\alpha\in[0,+\infty),
\end{equation}
then $q\gee 5$.
\end{rem}

\begin{lem}\label{lem8}
Assume $n\gee 5$, $m=3$ and $q>5$. Suppose $u\in C^{6}(\mathbb{R}^{n})$ is a positive radially symmetric entire solution to \eqref{PDE} such that $\liminf\limits_{|x|\rightarrow+\infty}\frac{u(x)}{|x|}\in(0,+\infty]$ and $\lim\limits_{r\rightarrow+\infty}u''(r)=\rho\in\mathbb{R}$, then $\rho\gee 0$. If $\rho>0$, then $u$ has exactly quadratic growth at $\infty$, that is,
\begin{equation}\label{e1-3}
  \lim _{|x| \rightarrow+\infty}\frac{u(x)}{|x|^{2}}=\frac{\rho}{2}\in(0,+\infty).
\end{equation}
If $\rho=0$, then $u$ has exactly linear growth at $\infty$, that is,
\begin{equation}\label{a30}
  \lim _{|x| \rightarrow+\infty}\frac{u(x)}{|x|}=-\frac{1}{n-1}\int_{0}^{+\infty}t^{2}\left(\Delta^{2}\bar{u}(t)-\frac{n-3}{t}\left(\Delta\bar{u}\right)'(t)\right)dt\in(0,+\infty).
\end{equation}
\end{lem}
\begin{proof}
By direct calculations, one has
\begin{equation}\label{12-3}
\bar{v}(r)={\Delta}^{2}\bar{u}(r)=\frac{1}{r^{n+1}}\left(r^{n+1}{\bar{u}}^{(3)}(r)\right)'(r)+\frac{n-3}{r}\bar{w}'(r), \qquad \forall \, r>0.
\end{equation}
By \eqref{PDE}, \eqref{12-3} and L'Hopital's rule, we have
\begin{eqnarray}\label{a26}
  && \quad \lim_{r\rightarrow+\infty}r{\bar{u}}^{(3)}(r)
  =\lim_{r\rightarrow+\infty}\frac{\int_{0}^{r}t^{n+1}\left(\bar{v}(t)-\frac{n-3}{t}\bar{w}'(t)\right)dt}{r^{n}} \\
  \nonumber && =\lim_{r\rightarrow+\infty}\frac{r^{2}}{n}\left(\bar{v}(r)-\frac{n-3}{r}\bar{w}'(r)\right)
  =-\frac{1}{2n(n-2)(n-4)}\lim_{r\rightarrow+\infty}r^{4}\overline{u^{-q}}(r)=0.
\end{eqnarray}
From \eqref{PDE}, \eqref{a26}, \eqref{12-3} and integrating, we can derive
\begin{eqnarray}\label{a25}
  && \quad \bar{u}(r)-\bar{u}(0) \\
 \nonumber &&=\frac{\rho}{2}r^{2}-\frac{1}{n(n-1)(n-2)r^{n-2}}
  \int_{0}^{r}t^{n+1}\left(\bar{v}(t)-\frac{n-3}{t}\bar{w}'(t)\right)dt \\
 \nonumber && \quad -\frac{r^{2}}{2n}\int_{r}^{+\infty}t\left(\bar{v}(t)-\frac{n-3}{t}\bar{w}'(t)\right)dt
  +\frac{1}{2(n-2)}\int_{0}^{r}t^{3}\left(\bar{v}(t)-\frac{n-3}{t}\bar{w}'(t)\right)dt \\
 \nonumber && \quad -\frac{r}{n-1}\int_{0}^{r}t^{2}\left(\bar{v}(t)-\frac{n-3}{t}\bar{w}'(t)\right)dt\\
 \nonumber && =:\frac{\rho}{2}r^{2}+\Psi(r), \qquad \forall \, r>0.
\end{eqnarray}
By L'Hopital's rule again, we get
\begin{eqnarray}\label{a27}
  && \lim_{r\rightarrow+\infty}\Psi'(r)=\lim_{r\rightarrow+\infty}\bigg[\frac{1}{n(n-1)r^{n-1}}
  \int_{0}^{r}t^{n+1}\left(\bar{v}(t)-\frac{n-3}{t}\bar{w}'(t)\right)dt \\
  \nonumber && \quad -\frac{r}{n}\int_{r}^{+\infty}t\left(\bar{v}(t)-\frac{n-3}{t}\bar{w}'(t)\right)dt
  -\frac{1}{n-1}\int_{0}^{r}t^{2}\left(\bar{v}(t)-\frac{n-3}{t}\bar{w}'(t)\right)dt\bigg] \\
  \nonumber &&=\frac{2-n}{(n-1)^{2}}\lim_{r\rightarrow+\infty}r^{3}\left(\bar{v}(r)-\frac{n-3}{r}\bar{w}'(r)\right)
  -\frac{1}{n-1}\int_{0}^{+\infty}t^{2}\left(\bar{v}(t)-\frac{n-3}{t}\bar{w}'(t)\right)dt \\
  \nonumber && =-\frac{1}{n-1}\int_{0}^{+\infty}t^{2}\left(\bar{v}(t)-\frac{n-3}{t}\bar{w}'(t)\right)dt.
\end{eqnarray}
Therefore, it follows from \eqref{a25} and \eqref{a27} that
\begin{equation}\label{a28}
  cr\lee \bar{u}(r)-\bar{u}(0)=\frac{\rho}{2}r^{2}-\left[\frac{1}{n-1}\int_{0}^{+\infty}t^{2}\left(\bar{v}(t)-\frac{n-3}{t}\bar{w}'(t)\right)dt\right]r+o(r),
\end{equation}
as $r\rightarrow+\infty$. Suppose $\rho<0$, then \eqref{a28} implies $\bar{u}(r)\rightarrow-\infty$ as $r\rightarrow+\infty$, which contradicts the fact $u>0$. Therefore, we must have $\rho\gee 0$. If $\rho>0$, then \eqref{a28} yields
\begin{equation}\label{a29}
  \lim _{|x| \rightarrow+\infty}\frac{u(x)}{|x|^{2}}=\frac{\rho}{2}\in(0,+\infty).
\end{equation}
If $\rho=0$, then \eqref{a28} yields
\begin{equation}\label{a30}
  \lim _{|x| \rightarrow+\infty}\frac{u(x)}{|x|}=-\frac{1}{n-1}\int_{0}^{+\infty}t^{2}\left(\bar{v}(t)-\frac{n-3}{t}\bar{w}'(t)\right)dt\in(0,+\infty).
\end{equation}
This completes our proof of Lemma \ref{lem8}.
\end{proof}
\begin{rem}\label{rem7}
The condition $\liminf\limits_{|x|\rightarrow+\infty}\frac{u(x)}{|x|}\in(0,+\infty]$ for radially symmetric positive entire solution $u$ in Lemma \ref{lem8} can be deduced from $u=o(|x|^{4})$ at $\infty$ when $n=5$ (see Lemma \ref{lem7}). In Lemma \ref{lem8}, we may assume ``$u''(r)>0$ and $u'''(r)<0$ for any $r>0$" instead of ``$\lim\limits_{r\rightarrow+\infty}u''(r)=\rho\in\mathbb{R}$ exists", since if ``$u''(r)>0$ and $u'''(r)<0$ for any $r>0$" holds then $\lim\limits_{r\rightarrow+\infty}u''(r)=\rho\gee 0$ exists.
\end{rem}

From now on, we will focus on the 5D case $n=5$. First, we can show that ``$\bar{u}''(r)>0$ and $\bar{u}'''(r)<0$ for any $r>0$" provided that the sub poly-harmonic property holds.
\begin{lem}\label{lem10}
Assume $n=5$, $m=3$ and $q\gee 2$. Suppose $u\in C^{6}(\mathbb{R}^{5})$ is a positive entire solution to \eqref{PDE} satisfying $u=o(|x|^{4})$ as $|x|\rightarrow+\infty$, then
\begin{equation}\label{a32}
  \bar{u}''(r)>0, \quad \bar{u}'''(r)<0, \qquad \forall \, r>0.
\end{equation}
Moreover, if assume further $q>5$ and $u$ is radially symmetric, then either $u$ has exactly quadratic growth at $\infty$ or $u$ has exactly linear growth at $\infty$.
\end{lem}
\begin{proof}
By \eqref{PDE} and integrating by parts, we have, for any $r>0$,
\begin{eqnarray}\label{a33}
   && \quad \bar{v}(r)-\frac{2}{r}\bar{w}'(r) \\
  \nonumber &&=\bar{w}''(r)+\frac{2}{r}\bar{w}'(r)=\frac{1}{r^{2}}\left(r^{2}\bar{w}'(r)\right)'(r) \\
 \nonumber &&=\frac{1}{r^{2}}\left(\frac{1}{r^{2}}\int_{0}^{r}t^{4}\bar{v}(t)dt\right)'(r)=\bar{v}(r)-\frac{2}{r^{5}}\int_{0}^{r}t^{4}\bar{v}(t)dt \\
 \nonumber &&= \frac{3}{5}\bar{v}(r)+\frac{2}{5r^{5}}\int_{0}^{r}t^{5}\bar{v}'(t)dt \\
 \nonumber &&= \frac{1}{5}\left(3\bar{v}(r)+r\bar{v}'(r)\right)-\frac{1}{5r^{5}}\int_{0}^{r}t^{6}\overline{u^{-q}}(t)dt.
\end{eqnarray}
Since $u=o(|x|^{4})$ at $\infty$, by Lemma \ref{lem5}, we have the sub poly-harmonic property: $w>0$ and $v<0$ in $\mathbb{R}^{5}$. By \eqref{a5}, \eqref{a6} and \eqref{a8} in Lemma \ref{lem7}, we have
\begin{equation}\label{a35}
  -\frac{10\Delta u(0)}{r^{2}}\lee \bar{v}(r)\lee -\frac{c_{1}}{r^{3}}, \quad u(r)\gee \frac{c_{2}}{10}r, \qquad \forall \, r\gee 1,
\end{equation}
where $c_{1}=-\max\limits_{|x|=1}v(x)>0$ and $c_{2}=\min\left\{\frac{-1}{2(n-4)}\max\limits_{|x|=1}v(x),\min\limits_{|x|=1}w(x)\right\}>0$. Moreover, from Lemma \ref{lem6} and \eqref{d5} in Corollary \ref{cor0}, we infer that there exists a constant $C>0$ such that
\begin{equation}\label{d4}
  0<\bar{v}'(r)\lee \frac{C}{r^{3}}, \qquad \forall \, r>0.
\end{equation}
By \eqref{d4}, we have
\begin{equation}\label{a36}
  0\lee \lim_{r\rightarrow+\infty}r\bar{v}'(r)\lee \lim_{r\rightarrow+\infty}\frac{C}{r^{2}}=0.
\end{equation}
Consequently, we conclude that
\begin{equation}\label{a37}
  3\bar{v}(r)+r\bar{v}'(r)\rightarrow 0, \qquad \text{as} \,\, r\rightarrow+\infty.
\end{equation}
Note that
\begin{equation}\label{a34}
  \left(3\bar{v}(r)+r\bar{v}'(r)\right)'(r)=r\left(\bar{v}''(r)+\frac{4}{r}\bar{v}'(r)\right)=r\overline{u^{-q}}(r)>0, \qquad \forall \, r>0,
\end{equation}
combining this with \eqref{a37}, we get
\begin{equation}\label{a38}
  3\bar{v}(r)+r\bar{v}'(r)<0, \qquad \forall \, r\gee 0.
\end{equation}
Subtracting \eqref{a38} into \eqref{a33}, it follows that
\begin{equation}\label{a39}
  \bar{v}(r)-\frac{2}{r}\bar{w}'(r)<0, \qquad \forall \, r>0.
\end{equation}
From \eqref{12-3} and \eqref{a39}, we derive
\begin{equation}\label{a40}
  \left(r^{6}{\bar{u}}^{(3)}(r)\right)'(r)<0, \qquad \forall \, r>0,
\end{equation}
and hence, by integrating, we get $\bar{u}'''(r)<0$ for all $r>0$. Now suppose on the contrary that the first inequality in \eqref{a32} does not hold, then there must exist $\eta>0$ and $r_{\ast}>0$ such that $\bar{u}''(r)\lee -\eta<0$ for all $r\gee r_{\ast}$. As a consequence, integrating yields that
\begin{equation}\label{a41}
  \bar{u}'(r)-\bar{u}'(r_{\ast})\lee -\eta(r-r_{\ast})\rightarrow-\infty, \qquad \text{as} \,\, r\rightarrow+\infty,
\end{equation}
which contradicts Lemma \ref{lem6}. Hence $\bar{u}''(r)>0$ for all $r>0$.

It follows from \eqref{a32} that $\lim\limits_{r\rightarrow+\infty}\bar{u}''(r)=:\rho\gee 0$ exists. If $u$ is radially symmetric, then the lower bound estimate in \eqref{a35} also implies that $\liminf\limits_{|x|\rightarrow+\infty}\frac{u(x)}{|x|}\in(0,+\infty]$. Therefore, from Lemma \ref{lem8}, we know that if $q>5$ and $u$ is radially symmetric, then either $\rho>0$ and hence $u$ has exactly quadratic growth at $\infty$, or $\rho=0$ and hence $u$ has exactly linear growth at $\infty$. This concludes our proof of Lemma \ref{lem10}.
\end{proof}
\begin{rem}\label{rem8}
In Lemma \ref{lem10}, we can also assume the sub poly-harmonic property ``$\Delta u>0$ and $\Delta^{2}u<0$ in $\mathbb{R}^{5}$" instead of ``$u=o(|x|^{4})$ at $\infty$", without changing its proof.
\end{rem}

\begin{lem}\label{lem9}
Assume $n=5$, $m=3$ and $q>0$. Suppose that the $C^{6}$ positive entire solution $u$ to \eqref{PDE} has at most linear growth (uniformly) at $\infty$, that is,
\begin{equation}\label{e9a-3a}
  \lim_{|x|\rightarrow+\infty}\frac{u(x)}{|x|}=\alpha\in[0,+\infty).
\end{equation}
Then, for any $x\in\mathbb{R}^{5}$,
\begin{equation}\label{a14}
  \Delta^{2}u(x)=-\frac{1}{8\pi^{2}}\int_{\mathbb{R}^{5}}\frac{1}{|x-y|^{3}}u^{-q}(y)dy,
\end{equation}
\begin{equation}\label{a15}
  \Delta u(x)=\frac{1}{16\pi^{2}}\int_{\mathbb{R}^{5}}\frac{1}{|x-y|}u^{-q}(y)dy.
\end{equation}
\end{lem}
\begin{proof}
From Corollary \ref{cor0} and Remark \ref{rem6}, we know that $\alpha>0$ and $q\gee 5$. Thus we may define the auxiliary functions
\begin{equation}\label{b0}
  I(x):=\frac{1}{8\pi^{2}}\int_{\mathbb{R}^{5}}\frac{1}{|x-y|^{3}}u^{-q}(y)dy, \qquad \forall \, x\in\mathbb{R}^{5},
\end{equation}
\begin{equation}\label{b1}
  J(x):=\frac{1}{16\pi^{2}}\int_{\mathbb{R}^{5}}\frac{1}{|x-y|}u^{-q}(y)dy, \qquad \forall \, x\in\mathbb{R}^{5}.
\end{equation}
It follows immediately that
\begin{equation}\label{b2}
  \Delta(\Delta^{2}u+I(x))=0, \quad \Delta(\Delta u-J(x))=\Delta^{2}u(x)+I(x), \qquad \forall \, x\in\mathbb{R}^{5}.
\end{equation}
By Lemma \ref{lem5}, we have the sub poly-harmonic property: $\Delta u>0$ and $\Delta^{2}u<0$ in $\mathbb{R}^{5}$. One can easily verify that, for $|x|$ sufficiently large,
\begin{eqnarray}\label{b3}
  && \frac{c_{1}}{|x|^{3}}:=\frac{4}{27}|x|^{-3}\int_{\mathbb{R}^{5}}u^{-q}(y)dy\lee \left(\frac{2}{3}\right)^{3}|x|^{-3}\int_{|y|<\frac{|x|}{2}}u^{-q}(y)dy \\
 \nonumber && \qquad \lee \int_{|y|<\frac{|x|}{2}}\frac{1}{|x-y|^{3}}u^{-q}(y)dy\lee \int_{\mathbb{R}^{5}}\frac{1}{|x-y|^{3}}u^{-q}(y)dy \\
 \nonumber && \qquad =I(x)\lee \frac{8}{|x|^{3}}\int_{|x-y|\gee \frac{|x|}{2}}u^{-q}(y)dy+\left(\frac{\alpha}{4}|x|\right)^{-q}\int_{|x-y|<\frac{|x|}{2}}\frac{1}{|x-y|^{3}}dy \\
 \nonumber && \qquad \lee \frac{8}{|x|^{3}}\int_{\mathbb{R}^{5}}u^{-q}(y)dy+\frac{\Sigma_{4}}{8}\left(\frac{\alpha}{4}\right)^{-q}|x|^{2-q}\lee \frac{c_{2}}{|x|^{3}},
\end{eqnarray}
where $\Sigma_{4}$ denotes the surface area of the unit $4$-sphere in $\mathbb{R}^{5}$. Similarly, we can also obtain that
\begin{equation}\label{b4}
  \frac{\tilde{c}_{1}}{|x|}\lee J(x)\lee \frac{\tilde{c}_{2}}{|x|}, \qquad \forall \, |x| \,\, \text{large enough.}
\end{equation}
Therefore, Liouville theorem implies that
\begin{equation}\label{b5}
  \Delta^{2}u(x)+I(x)=C_{0}\lee 0, \qquad \forall \, x\in\mathbb{R}^{5}.
\end{equation}
Taking spherical average, by \eqref{a5} and \eqref{a6} in Lemma \ref{lem7} and \eqref{b3}, we have
\begin{equation}\label{b6}
  C_{0}=\Delta^{2}\bar{u}(r)+\bar{I}(r)=\bar{v}(r)+\bar{I}(r)\rightarrow0, \qquad \text{as} \,\, r\rightarrow+\infty,
\end{equation}
and hence
\begin{equation}\label{b7}
  \Delta^{2}u(x)+I(x)=C_{0}=0, \qquad \forall \, x\in\mathbb{R}^{5}.
\end{equation}

By \eqref{b2}, \eqref{b4} and \eqref{b7}, we deduce from Liouville theorem that
\begin{equation}\label{b8}
  \Delta u(x)-J(x)=C_{\infty}\gee 0, \qquad \forall \, x\in\mathbb{R}^{5}.
\end{equation}
If $C_{\infty}>0$, then
\begin{equation}\label{b9}
  \lim_{r\rightarrow+\infty}\bar{w}(r)=\lim_{r\rightarrow+\infty}\Delta\bar{u}(r)
  =\lim_{r\rightarrow+\infty}\frac{1}{r^{4}}\left(r^{4}\bar{u}'(r)\right)'(r)=C_{\infty}>0.
\end{equation}
By integrating twice, we infer from \eqref{b9} that, for $r>0$ sufficiently large,
\begin{equation}\label{b10}
  \bar{u}(r)\gee \frac{C_{\infty}}{32}r^{2},
\end{equation}
which contradicts the exact linear growth \eqref{e9a-3a} of $u$ at $\infty$. Thus $C_{\infty}=0$ and
\begin{equation}\label{b11}
  \Delta u(x)-J(x)=0, \qquad \forall \, x\in\mathbb{R}^{5}.
\end{equation}
This finishes our proof of Lemma \ref{lem9}.
\end{proof}

We can show the following integral representation formula for $C^{6}$ positive entire solution $u$ to \eqref{PDE} that has exactly linear growth (uniformly) at $\infty$ and hence derive the classification result when $q=11$.
\begin{thm}\label{thm9}
Assume $n=5$, $m=3$ and $q>6$. Suppose the $C^{6}$ positive entire solution $u$ to \eqref{PDE} has at most linear growth (uniformly) at $\infty$, that is,
\begin{equation}\label{e9a-3l}
  \lim_{|x|\rightarrow+\infty}\frac{u(x)}{|x|}=\alpha\in[0,+\infty).
\end{equation}
Then $u$ has the following integral representation:
\begin{equation}\label{e5-3}
  u(x)=\frac{1}{64\pi^{2}}\int_{\mathbb{R}^{5}}|x-y|u^{-q}(y)dy+\gamma, \qquad \forall \, x\in\mathbb{R}^{5},
\end{equation}
where $\gamma$ is a constant. Moreover, $\gamma<0$ if $6<q<11$, $\gamma=0$ if $q=11$ and $\gamma>0$ if $q>11$. Furthermore, if $q=11$, then, up to dilations and translations, $u$ must assume the unique form: $u(x)=c(1+|x|^{2})^{\frac{1}{2}}$.
\end{thm}
\begin{proof}
From Corollary \ref{cor0}, we know that $\alpha>0$. Since $q>6$, we may define the auxiliary function
\begin{equation}\label{c0}
  K(x):=\frac{1}{64\pi^{2}}\int_{\mathbb{R}^{5}}|x-y|u^{-q}(y)dy, \qquad \forall \, x\in\mathbb{R}^{5}.
\end{equation}
By direct calculation and Lemma \ref{lem9}, one can easily verify that
\begin{equation}\label{c1}
  \Delta\left(u-K\right)(x)=\Delta u(x)-\frac{1}{16\pi^{2}}\int_{\mathbb{R}^{5}}\frac{1}{|x-y|}u^{-q}(y)dy=0, \qquad \forall \, x\in\mathbb{R}^{5}.
\end{equation}
Let
\begin{equation}\label{c2}
  \frac{1}{64\pi^{2}}\int_{\mathbb{R}^{5}}u^{-q}(y)dy=:\zeta\in(0,+\infty).
\end{equation}
Since $|x|u^{-q}\in L^{1}(\mathbb{R}^{5})$, one can deduce from Lebesgue's dominated convergence theorem that
\begin{equation}\label{c3}
  \lim_{|x|\rightarrow+\infty}\frac{K(x)}{|x|}=\zeta\in(0,+\infty).
\end{equation}
Combining \eqref{e9a-3l}, \eqref{c1} and \eqref{c3}, it follows from Liouville theorem that
\begin{equation}\label{c4}
  u(x)=K(x)+\sum_{k=1}^{5}a_{k}x_{k}+\gamma, \qquad \forall \, x\in\mathbb{R}^{5}
\end{equation}
for some constants $a_{k}$ ($1\lee k\lee 5$) and $\gamma$. Dividing \eqref{c4} by $|x|$, we get
\begin{equation}\label{c5}
  \frac{u(x)}{|x|}=\frac{K(x)}{|x|}+\vec{a}\cdot\vec{e}+\frac{\gamma}{|x|}, \qquad \forall \, x\in\mathbb{R}^{5}\setminus\{0\},
\end{equation}
where the unit vector $\vec{e}:=\frac{x}{|x|}\in\partial B_{1}(0)$. Now letting $|x|\rightarrow+\infty$ and taking limits in \eqref{c5}, combining with \eqref{e9a-3l} and \eqref{c3} yield that
\begin{equation}\label{c6}
  \alpha=\zeta+\vec{a}\cdot\vec{e}, \qquad \forall \, \vec{e}\in\partial B_{1}(0).
\end{equation}
Hence we have $\alpha=\zeta$ and $a_{k}=0$ for every $1\lee k\lee 5$. Consequently,
\begin{equation}\label{c7}
  u(x)=K(x)+\gamma=\frac{1}{64\pi^{2}}\int_{\mathbb{R}^{5}}|x-y|u^{-q}(y)dy+\gamma, \qquad \forall \, x\in\mathbb{R}^{5}.
\end{equation}

\medskip

We can extend the integral representation \eqref{e5-3} from $q>6$ to $q>5$ and prove the following Lemma.
\begin{lem}\label{lem11}
Assume $n=5$, $m=3$ and $q>5$. Suppose the $C^{6}$ positive entire solution $u$ to \eqref{PDE} has at most linear growth (uniformly) at $\infty$, that is,
\begin{equation}\label{c8}
  \lim_{|x|\rightarrow+\infty}\frac{u(x)}{|x|}=\alpha\in[0,+\infty).
\end{equation}
Then $u$ has the following integral representation:
\begin{equation}\label{c9}
  u(x)=\frac{1}{64\pi^{2}}\int_{\mathbb{R}^{5}}\left[|x-y|-|y|\right]u^{-q}(y)dy+\widetilde{\gamma}, \qquad \forall \, x\in\mathbb{R}^{5},
\end{equation}
where $\widetilde{\gamma}$ is a constant. Consequently,
\begin{equation}\label{c10}
  \nabla u(x)=\frac{1}{64\pi^{2}}\int_{\mathbb{R}^{5}}\frac{x-y}{|x-y|}u^{-q}(y)dy, \qquad \forall \, x\in\mathbb{R}^{5}.
\end{equation}
\end{lem}
\begin{proof}
From Corollary \ref{cor0}, we know that $\alpha>0$. When $q>6$, we have done. Now we only consider the case $5<q\lee 6$. Define the auxiliary function
\begin{equation}\label{c11}
  \widetilde{K}(x):=\frac{1}{64\pi^{2}}\int_{\mathbb{R}^{5}}\left[|x-y|-|y|\right]u^{-q}(y)dy, \qquad \forall \, x\in\mathbb{R}^{5}.
\end{equation}
It is clear from Lemma \ref{lem9} that
\begin{equation}\label{c12}
\Delta\left(u-\widetilde{K}\right)(x)=\Delta u(x)-\frac{1}{16\pi^{2}}\int_{\mathbb{R}^{5}}\frac{1}{|x-y|}u^{-q}(y)dy=0, \qquad \forall \, x\in\mathbb{R}^{5}.
\end{equation}
Since $\zeta:=\int_{\mathbb{R}^{5}}u^{-q}(x)dx<+\infty$, one can deduce from Lebesgue's dominated convergence theorem that
\begin{equation}\label{c13}
  \lim_{|x|\rightarrow+\infty}\frac{\widetilde{K}(x)}{|x|}=\zeta\in(0,+\infty).
\end{equation}
The rest of the proof is entirely the same to the above proof for the case $q>6$, and we can finally derive from Liouville theorem that
\begin{equation}\label{c14}
  u(x)=\widetilde{K}(x)+\widetilde{\gamma}, \qquad \forall \, x\in\mathbb{R}^{5}
\end{equation}
for some constant $\widetilde{\gamma}$. Hence we finishes our proof of Lemma \ref{lem11}.
\end{proof}

\medskip

Now we continue carrying out our proof of Theorem \ref{thm9}. By \eqref{c10} in Lemma \ref{lem11}, we have
\begin{equation}\label{c15}
 x\cdot\nabla u(x)=\frac{1}{64\pi^{2}}\int_{\mathbb{R}^{5}}\frac{|x|^{2}-x\cdot y}{|x-y|}u^{-q}(y)dy, \qquad \forall \, x\in\mathbb{R}^{5}.
\end{equation}
Multiplying \eqref{c15} by $u^{-q}$ and integrating over $B_{R}(0)$ yield that
\begin{equation}\label{c16}
  \int_{B_{R}(0)}\frac{1}{1-q}x\cdot\left[\nabla u^{1-q}\right](x)dx=\frac{1}{64\pi^{2}}\int_{\mathbb{R}^{5}}\left[\int_{B_{R}(0)}\frac{|x|^{2}-x\cdot y}{|x-y|}u^{-q}(x)dx\right]u^{-q}(y)dy
\end{equation}
for any $R>0$. Integrating by parts, due to \eqref{e9a-3l} and $q>6$, we get
\begin{eqnarray}\label{c17}
  && \int_{B_{R}(0)}\frac{1}{1-q}x\cdot\left[\nabla u^{1-q}\right](x)dx
  =\frac{R}{1-q}\int_{\partial B_{R}(0)}u^{1-q}(x)d\sigma-\frac{5}{1-q}\int_{B_{R}(0)}u^{1-q}(x)dx \\
 \nonumber &&=o_{R}(1)-\frac{5}{1-q}\int_{B_{R}(0)}u^{1-q}(x)dx,
\end{eqnarray}
as $R\rightarrow+\infty$. At the same time, by the integral representation formula \eqref{e5-3}, we deduce from \eqref{e9a-3l} and $q>6$ that
\begin{eqnarray}\label{c18}
  && \frac{1}{64\pi^{2}}\int_{\mathbb{R}^{5}}\left[\int_{B_{R}(0)}\frac{|x|^{2}-x\cdot y}{|x-y|}u^{-q}(x)dx\right]u^{-q}(y)dy \\
 \nonumber &&=\frac{1}{128\pi^{2}}\int_{\mathbb{R}^{5}}\left[\int_{B_{R}(0)}\frac{|x-y|^{2}+|x|^{2}-|y|^{2}}{|x-y|}u^{-q}(x)dx\right]u^{-q}(y)dy  \\
 \nonumber &&=\frac{1}{2}\int_{B_{R}(0)}u^{-q}(x)\left[u(x)-\gamma\right]dx
  +\frac{1}{128\pi^{2}}\int_{\mathbb{R}^{5}}\left[\int_{B_{R}(0)}\frac{|x|^{2}-|y|^{2}}{|x-y|}u^{-q}(x)dx\right]u^{-q}(y)dy  \\
 \nonumber && =\frac{1}{2}\int_{B_{R}(0)}u^{1-q}(x)dx-\frac{\gamma}{2}\int_{B_{R}(0)}u^{-q}(x)dx+o_{R}(1),
\end{eqnarray}
as $R\rightarrow+\infty$. Now letting $R\rightarrow+\infty$ and taking limits in \eqref{c16}, combining with \eqref{c17} and \eqref{c18}, we derive
\begin{equation}\label{c19}
  \frac{5}{q-1}\int_{\mathbb{R}^{5}}u^{1-q}(x)dx=\frac{1}{2}\int_{\mathbb{R}^{5}}u^{1-q}(x)dx-\frac{\gamma}{2}\int_{\mathbb{R}^{5}}u^{-q}(x)dx,
\end{equation}
and hence
\begin{equation}\label{c20}
  \frac{11-q}{2(q-1)}\int_{\mathbb{R}^{5}}u^{1-q}(x)dx=-\frac{\gamma}{2}\int_{\mathbb{R}^{5}}u^{-q}(x)dx.
\end{equation}
As a consequence, we deduce that $\gamma<0$ if $6<q<11$, $\gamma=0$ if $q=11$, and $\gamma>0$ if $q>11$. In particular, when $q=11$, we have $\gamma=0$ and hence $u\in C^{6}(\mathbb{R}^{5})$ is a positive entire solution to the integral equation:
\begin{equation}\label{IE}
  u(x)=\frac{1}{64\pi^{2}}\int_{\mathbb{R}^{5}}|x-y|u^{-11}(y)dy, \qquad \forall \, x\in\mathbb{R}^{5}.
\end{equation}
It follows from Feng and Xu \cite{FX} that, up to dilations and translations, $u$ must assume the unique form: $u(x)=c(1+|x|^{2})^{\frac{1}{2}}$. This concludes our proof of Theorem \ref{thm9}.
\end{proof}

In the following theorem, we mainly concern about asymptotic properties and nonexistence for radially symmetric positive entire solutions $u$ without the property $u(x)=o(|x|^{4})$ at $\infty$ or the sub poly-harmonic property.
\begin{thm}\label{thm10}
i) \, Assume $n\gee 3$ but $n\neq4, 6$ and $q>\frac{1}{2}$. Suppose $u$ is a radially symmetric positive entire solution to \eqref{PDE} satisfying $\liminf\limits_{|x|\rightarrow+\infty}\frac{u(x)}{|x|^{4}}\in(0,+\infty]$, then $\lim\limits_{|x|\rightarrow+\infty}\frac{u(x)}{|x|^{4}}=\frac{1}{8n(n+2)}\lim\limits_{|x|\rightarrow+\infty}\Delta^{2}u(x)\in(0,+\infty)$. That is, if radially symmetric positive solution $u$ in entire $\mathbb{R}^{n}$ has at least quartic growth at $\infty$ then it must has exactly quartic growth at $\infty$. \\
ii) \, Conversely, for any $n\gee 2$ and $q>0$, suppose $u$ is a radially symmetric positive entire solution to \eqref{PDE}. We have: (a) If $\lim\limits_{|x|\rightarrow+\infty}\Delta^{2}u(x)\in(0,+\infty]$, then $\liminf\limits_{|x|\rightarrow+\infty}\frac{u(x)}{|x|^{4}}\in(0,+\infty]$. Furthermore, if assume further $n\gee 3$ but $n\neq4, 6$ and $q>\frac{1}{2}$, then $\lim\limits_{|x|\rightarrow+\infty}\frac{u(x)}{|x|^{4}}=\frac{1}{8n(n+2)}\lim\limits_{|x|\rightarrow+\infty}\Delta^{2}u(x)\in(0,+\infty)$. (b) If $\lim\limits_{|x|\rightarrow+\infty}\Delta^{2}u(x)=0$, then $u$ satisfy the sub poly-harmonic property and hence has at most quadratic growth at $\infty$.
\end{thm}
\begin{proof}
(i) Recall that $v:=\Delta^{2}u$ and $w=\Delta u$. By the tri-harmonic equation \eqref{PDE} and integrating by parts, we can derive, for $n=3,5$ or $n\gee 7$,
\begin{equation}\label{f0}
\begin{aligned}
& \quad \bar{u}(r)-\bar{u}(0) \\
&=\frac{1}{8n(n+2)}\left(\bar{v}(0)+\frac{1}{n-2}\int_{0}^{r}t\overline{u^{-q}}(t)dt\right)r^{4} \\
&\quad +\frac{1}{2n}\left(\bar{w}(0)-\frac{1}{2(n-2)(n-4)}\int_{0}^{r}t^{3}\overline{u^{-q}}dt\right)r^{2}  \\
&\quad +\frac{1}{8(n-2)(n-4)(n-6)}\left(\int_{0}^{r}t^{5}\overline{u^{-q}}(t)dt-\frac{1}{r^{n-6}}\int_{0}^{r}t^{n-1}\overline{u^{-q}}(t)dt\right) \\
&\quad +\frac{1}{4n(n-2)}\left(\frac{1}{(n-4)r^{n-4}}\int_{0}^{r}t^{n+1}\overline{u^{-q}}(t)dt
-\frac{1}{2(n+2)r^{n-2}}\int_{0}^{r}t^{n+3}\overline{u^{-q}}(t)dt\right).
\end{aligned}
\end{equation}
As a consequence, we can rewrite \eqref{f0} into the following form:
\begin{equation}\label{f1}
\bar{u}(r)-\bar{u}(0)=-\frac{n+4}{8n^{2}(n+2)}\bar{v}(r)r^{4}+\frac{\bar{w}(r)}{2n}r^{2}+\Gamma(r),
\end{equation}
where
\begin{equation}\label{f2}
\begin{aligned}
&\Gamma(r):=\frac{1}{8(n-2)}\left(\frac{1}{(n-6)(n-4)}\int_{0}^{r}t^{5}\overline{u^{-q}}(t)dt
-\frac{1}{n(n+2)r^{n-2}}\int_{0}^{r}t^{n+3}\overline{u^{-q}}(t)dt\right) \\
&\qquad +\frac{1}{2n^{2}}\left(\frac{1}{(n-4)r^{n-4}}\int_{0}^{r}t^{n+1}\overline{u^{-q}}(t)dt
-\frac{n+3}{(n+2)(n-6)r^{n-6}}\int_{0}^{r}t^{n-1}\overline{u^{-q}}(t)dt\right).
\end{aligned}
\end{equation}
Note that by equation \ref{PDE}, $\bar{v}'(r)>0$ and hence $\bar{v}(r)$ is strictly increasing with respect to $r$ and $\lim\limits_{r\rightarrow+\infty}\bar{v}(r)\in(-\infty,+\infty]$. Now suppose $q>\frac{1}{2}$ and $u$ is a radially symmetric positive entire solution satisfying $\liminf\limits_{|x|\rightarrow+\infty}\frac{u(x)}{|x|^{4}}\in(0,+\infty]$, by L'Hopital's rule, we have, for any fixed $\delta\in(0,2q-1)$ small,
\begin{equation}\label{f6}
  \lim_{r\rightarrow+\infty}r^{1+\delta}\frac{\int_{0}^{r}t^{n-1}u^{-q}(t)dt}{r^{n-1}}
  =\frac{1}{n-2-\delta}\lim_{r\rightarrow+\infty}r^{2+\delta}u^{-q}(r)\lee C\lim_{r\rightarrow+\infty}r^{1-2q}=0,
\end{equation}
and hence
\begin{equation}\label{f7}
  \lim_{r\rightarrow+\infty}v(r)=v(0)+\lim_{r\rightarrow+\infty}\int_{0}^{r}\frac{\int_{0}^{t}s^{n-1}u^{-q}(s)ds}{t^{n-1}}dt<+\infty.
\end{equation}
By L'Hopital's rule, we obtain
\begin{equation}\label{f3}
  0\lee \lim_{r\rightarrow+\infty}\left|\frac{\Gamma(r)}{r^{4}}\right|\lee C_{n}\lim_{r\rightarrow+\infty}r^{2}u^{-q}(r)\lee C\lim_{r\rightarrow+\infty}r^{2-4q}=0,
\end{equation}
\begin{equation}\label{f4}
  \lim_{r\rightarrow+\infty}\frac{w(r)}{r^{2}}=\lim_{r\rightarrow+\infty}\frac{w'(r)}{2r}=\lim_{r\rightarrow+\infty}\frac{\int_{0}^{r}t^{n-1}v(t)dt}{2r^{n}}
  =\frac{1}{2n}\lim_{r\rightarrow+\infty}v(r)\in(-\infty,+\infty).
\end{equation}
Combining \eqref{f3}, \eqref{f4} with \eqref{f1}, we derive
\begin{equation}\label{f5}
  \lim_{r\rightarrow+\infty}\frac{u(r)}{r^{4}}=\frac{1}{8n(n+2)}\lim_{r\rightarrow+\infty}v(r).
\end{equation}
Therefore, we can infer from $\liminf\limits_{|x|\rightarrow+\infty}\frac{u(x)}{|x|^{4}}\in(0,+\infty]$, \eqref{f7} and \eqref{f5} that
\begin{equation}\label{f8}
  \lim_{r\rightarrow+\infty}\frac{u(r)}{r^{4}}=\frac{1}{8n(n+2)}\lim_{r\rightarrow+\infty}\Delta^{2}u(r)\in(0,+\infty).
\end{equation}

\medskip

(ii) Conversely, assume $n\gee 2$, $q>0$ and the positive radially symmetric entire solution $u$ satisfies $\lim\limits_{r\rightarrow+\infty}\Delta^{2}u(r)\in(0,+\infty]$, then by integrating twice, one gets $w(r)\gee cr^{2}$ for some constant $c>0$ and all $r$ large enough. By integrating twice again, we arrive at $u(r)\gee cr^{4}$ for $r$ large enough, that is, $\liminf\limits_{|x|\rightarrow+\infty}\frac{u(x)}{|x|^{4}}\in(0,+\infty]$. From (i), we get, if $n\gee 3$ but $n\neq4, 6$ and $q>\frac{1}{2}$, then $\lim\limits_{|x|\rightarrow+\infty}\frac{u(x)}{|x|^{4}}=\frac{1}{8n(n+2)}\lim\limits_{r\rightarrow+\infty}\Delta^{2}u(r)\in(0,+\infty)$.

Next, suppose $n\gee 2$, $q>0$ and the positive radially symmetric entire solution $u$ satisfies $\lim\limits_{|x|\rightarrow+\infty}\Delta^{2}u(x)=0$, it follows immediately from $v'(r)>0$ and Lemma 3.3 in \cite{NNPY} that $v(r)<0$ and $w(r)>0$ for any $r\gee 0$. Thus the sub poly-harmonic property holds. By the upper bound estimate \eqref{7-3a}, we know $u(r)\lee Cr^{2}$ for all $r$ large enough, that means $u$ has at most quadratic growth at $\infty$. This concludes our proof of Theorem \ref{thm10}.
\end{proof}

This concludes our proof of Theorem \ref{thm1}.

\section{2D bi-harmonic equations with negative exponents}

In this section, we will prove the nonexistence of positive solutions to the 2D bi-harmonic equation \eqref{PDE}, i.e.,  Theorem \ref{thm0}.

From Theorem 1 in \cite{N}, we can derive the following sub poly-harmonic property for $C^{4}$ positive solution $u$ to the 2D bi-harmonic equation \eqref{PDE}.
\begin{lem}\label{lem1}
Assume $n=2$, $m=2$ and $q>0$. If $u\in C^{4}(\mathbb{R}^{2})$ is a positive entire solution in $\mathbb{R}^{2}$ to \eqref{PDE}, then $\Delta u>0$ in $\mathbb{R}^{2}$.
\end{lem}

Now we define $w:=\Delta u$ and
\begin{equation}\label{3}
\bar{u}(r):=\ \ \  -\kern-20.5pt\int\limits_{\p B_{r}(0)} u\,d\sigma, \qquad \bar{w}(r):=\ \ \  -\kern-20.5pt\int\limits_{\p B_{r}(0)}\Delta u\,d\sigma, \qquad \forall \, r\gee0.
\end{equation}
Recall that, in $\mathbb{R}^{2}$, for any radially symmetric function $f(r)$, $\Delta f(r)=\frac{1}{r}\left(rf'\right)'$. It can be deduced from Lemma \ref{lem0} that $\bar{u}(r)$ and $\bar{w}(r)$ satisfy
\begin{align}\label{5}
\left\{
\begin{aligned}
&\Delta\bar{u}(r)=\bar{w}(r), \quad \forall \, r\gee 0, \\
&\Delta\bar{w}(r)+{\bar{u}}^{-q}(r)\lee 0, \quad \forall \, r\gee 0.
\end{aligned}
\right.
\end{align}

We have the following lemma.
\begin{lem}\label{lem2}
Assume $n=2$, $m=2$ and $q>0$. If $u\in C^{4}(\mathbb{R}^{2})$ is a positive entire solution in $\mathbb{R}^{2}$ to \eqref{PDE}, then for all $r>0$,
\begin{equation}\label{8}
\bar{u}'(r)>0, \quad \bar{u}''(r)>0, \quad \bar{u}'''(r)<0,
\end{equation}
\begin{equation}\label{9}
\bar{w}'(r)<0.
\end{equation}
\end{lem}
\begin{proof}
We multiply the inequality in \eqref{5} by $r$ and integrate the resulting equation to get
\begin{equation}\label{e12}
r\bar{w}'(r) + \int_{0}^{r} t{\bar{u}}^{-q}dt \lee 0.
\end{equation}
Then \eqref{9} follows immediately. Similarly, from the facts that $w>0$ and $(r\bar{u}')'=r\bar{w}$, we can deduce by integrating the first inequality in \eqref{8}.

Next, since we have
\begin{align}\label{10}
{\Delta}^{2}\bar{u}(r) &={\bar{u}}^{(4)}+\frac{2}{r}{\bar{u}}^{(3)}-\frac{1}{r^{2}}{\bar{u}}^{(2)}+\frac{1}{r^{3}}\bar{u}' \\
&=\frac{1}{r^{2}}(r^{2}{\bar{u}}^{(3)})'-\frac{1}{r^{2}}{\bar{u}}^{(2)}+\frac{1}{r^{3}}\bar{u}'\notag\\
&<0 \notag
\end{align}
and
\begin{equation}\label{11}
\bar{w}'(r)={\bar{u}}^{(3)}+\frac{1}{r}{\bar{u}}^{(2)}-\frac{1}{r^{2}}\bar{u}'<0,
\end{equation}
it follows that
\begin{equation}\label{12}
{\Delta}^{2}\bar{u}(r)+\frac{1}{r}\bar{w}'(r)=\frac{1}{r^{3}}(r^{3}{\bar{u}}^{(3)})'<0.
\end{equation}
Hence $(r^{3}{\bar{u}}^{(3)})'<0$ for all $r>0$. An integration gives the third inequality in \eqref{8}. This implies that ${\bar{u}}^{(2)}(r)$ strictly decreases with respect to $r\in(0,+\infty)$. Suppose that the second inequality in \eqref{8} does not hold, then there exists a $r_{0}>0$ such that ${\bar{u}}^{(2)}(r_{0})\lee0$. Then there exist $\delta>0$ and $r_{1}>r_{0}$ such that for any $r\gee r_{1}$, ${\bar{u}}^{(2)}(r)\lee -\delta$. An integration yields
\begin{equation}\label{e13}
\bar{u}'(r)-\bar{u}'(r_{1}) \lee -\delta(r-r_{1})
\end{equation}
for any $r \gee r_{1}$. But this implies $\bar{u}'<0$ for $r>0$ large enough, which contradicts the first inequality in \eqref{8}. Hence the second inequality in \eqref{8} holds. This finishes our proof of Lemma \ref{lem2}.
\end{proof}

By Lemma \ref{lem2}, we have $\bar{w}'(r)<0$ for any $r>0$, and hence $\bar{w}(r)=\Delta\bar{u}(r)\lee \bar{w}(0)=w(0)=\Delta u(0)$. Now by integrating again, we arrive at
\begin{equation}\label{7}
\bar{u}(r)\lee \bar{u}(0)+\frac{\bar{w}(0)}{4}r^{2}=u(0)+\frac{\Delta u(0)}{4}r^{2}, \qquad \forall \, r\gee 0.
\end{equation}
Consequently, any $C^{4}$ positive entire solution $u$ to the 2D bi-harmonic equation \eqref{PDE} must satisfy
\begin{equation}\label{e11}
  \liminf_{|x|\rightarrow+\infty}\frac{u(x)}{|x|^{2}}\lee\frac{\Delta u(0)}{4}<+\infty.
\end{equation}

We will prove the following comparison theorem in 2 dimension, which plays a crucial role in the proof of Theorem \ref{thm0} and is quite interesting itself. For more comparison results on quasi-monotone ODE systems, please refer to Walter [8] (cf. \cite{DFG,MR} for 3D comparison theorems, see also \cite{FF}).
\begin{thm}[2D comparison theorem]\label{thm2}
Assume the spatial dimension $n=2$ and $q>0$. If radially symmetric functions $U, \, V \in C^{4}(\mathbb{R}^{2})$ satisfy $U>0$, $V>0$ and
\begin{align}
\left\{
\begin{aligned}
&\Delta^{2}(U-V)(r) \gee V^{-q}(r)-U^{-q}(r),\ \ \ \   \forall \, r \gee 0,\\
&U(0) = V(0), \\
&U^{(k)}(0) = V^{(k)}(0),\ \ \ \ k=1,2,3.
\end{aligned}
\right.
\end{align}
Then we have $U(r)\gee V(r)$ for all r $\gee 0$. Furthermore, if $\Delta^{2}(U-V)(0)>0$, then
\begin{equation}
U(r)>V(r),\ \ \ \ \forall\ r>0.
\end{equation}
\end{thm}
\begin{proof}
By direct calculations, one has, for any $r>0$,
\begin{equation}
{\Delta}^{2}(U-V)(r)={(U-V)}^{(4)}(r)+\frac{2}{r}{(U-V)}^{(3)}(r)-\frac{1}{r^{2}}{(U-V)}^{(2)}(r)+\frac{1}{r^{3}}(U-V)'(r).
\end{equation}
Thus we have
\begin{align}
0&=V^{-q}(0)-U^{-q}(0)\\
\nonumber &\lee {\Delta}^{2}(U-V)(0)\\
\nonumber &= \lim_{r\rt 0}{\Delta}^{2}(U-V)(r)\\
\nonumber &=\frac{8}{3}\lim_{r\rt 0}(U-V)^{(4)}(r)\\
\nonumber &=\frac{8}{3}(U-V)^{(4)}(0).
\end{align}
That is, $ (U-V)^{(4)}(0) \gee 0$. For arbitrarily given $\varepsilon \in(0,1)$, we have
\begin{align}
U(r)-V(r)+\varepsilon r^{4}&=\frac{(U-V)^{(4)}(\theta r)}{4!}r^{4}+\varepsilon r^{4} \\
\nonumber &=\left[\frac{(U-V)^{(4)}(\theta r)}{4!}+\varepsilon\right]r^{4}
\end{align}
for some $0<\theta<1$. It follows immediately that there exists a $r(\varepsilon)>0$ small enough such that
\begin{equation}
U(r)-V(r)+\varepsilon r^{4} \gee 0,\ \ \ \ \forall\ 0\lee r\lee r(\varepsilon).
\end{equation}

\smallskip

Now let $r_{*}(\varepsilon):=\sup\{r\ |\ U(s)-V(s)+\varepsilon s^{4} \gee 0,\ \forall\ 0\lee s\lee r\}$, then $r_{*}(\varepsilon)>0$. There exists a $0<r_{0}\lee\min\left\{\sqrt[4]{\frac{V(0)}{4}},2\sqrt[4]{\frac{V(0)^{q+1}}{2^{2q+1}q}}\right\}$ (independent of $\varepsilon$) sufficiently small such that
\begin{equation}
V(r)\gee\frac{V(0)}{2}, \,\quad\, V(r)-\varepsilon r^{4}\gee\frac{V(0)}{4}, \qquad \forall\ 0\lee r\lee r_{0}.
\end{equation}
Therefore, by the Mean Value Theorem, we have, for any $0\lee r\lee \min\{r_{*}(\varepsilon), r_{0}\}=:{\bar{r}}(\varepsilon)$,
\begin{align}\label{e15}
{\Delta}^{2}(U(r)-V(r)+\varepsilon r^{4})&\gee 64\varepsilon+V^{-q}(r)-U^{-q}(r)\\
\nonumber &\gee 64\varepsilon + V^{-q}(r)-\left(V(r)-\varepsilon r^{4}\right)^{-q}\\
\nonumber &=64\varepsilon-q\xi(r)^{-q-1}\varepsilon r^{4}\\
\nonumber &\gee \varepsilon\left[64-q\left(V(r)-\varepsilon r^{4}\right)^{-q-1}r^{4}\right] \\
\nonumber &\gee \varepsilon\left[64-\frac{q4^{q+1}r^{4}}{V(0)^{q+1}}\right]\gee 32\varepsilon>0,
\end{align}
where $\xi(r)\in(V(r)-\varepsilon r^{4},V(r))$. By direct calculations, we also have
\begin{align}\label{e14}
\Delta^{2}(U-V+\varepsilon r^{4})&=\frac{1}{r^{3}}\left(r^{3}(U-V+\varepsilon r^{4})^{(3)}\right)'-\frac{1}{r}\left[\Delta(U-V+\varepsilon r^{4})\right]'(r)\\
\nonumber &=\frac{1}{r}\left(r\left[\Delta(U-V+\varepsilon r^{4})\right]'\right)', \qquad \forall \, r>0.
\end{align}
As a consequence, we can infer from \eqref{e15} and \eqref{e14} that
\begin{equation}
[\Delta(U-V+\varepsilon r^{4})]'(r)>0,\ \ \ \ \forall\ r\in(0,\bar{r}(\varepsilon)].
\end{equation}
Combining this with \eqref{e14} implies that
\begin{equation}
(r^{3}(U-V+\varepsilon r^{4})^{(3)})'(r)>0,\ \ \ \ \forall\ r\in(0,\bar{r}(\varepsilon)].
\end{equation}
By integrating, we can deduce further that
\begin{equation}
 (U-V+\varepsilon r^{4})^{(k)}(r)>0,\ \ \ \ k=1,2,3, \quad \forall\ r\in(0,\bar{r}(\varepsilon)],
\end{equation}
and hence
\begin{equation}
(U-V)(\bar{r}(\varepsilon))+\varepsilon \bar{r}^{4}(\varepsilon)>0.
\end{equation}
If $r_{*}(\varepsilon)<r_{0}$, then $\bar{r}(\varepsilon)=r_{*}(\varepsilon)$ and hence $\left(U-V+\varepsilon r^{4}\right)(\bar{r}(\varepsilon))=0$. This is a contradiction. Thus we must have $r_{*}(\varepsilon)\gee r_{0}>0$ for all $\varepsilon \in (0,1)$, that is,
\begin{equation}
U(r)-V(r)+\varepsilon r^{4} \gee 0, \qquad \forall \, 0\lee r\lee r_{0}, \quad \forall \, 0<\varepsilon<1.
\end{equation}

\smallskip

Let $\varepsilon \rt 0+$, then we can obtain $U(r)\gee V(r)$ for any $0\lee r\lee r_{0}$. Define $r_{*}:=\sup\{r\ |\ U(s)\gee V(s),\ \ \forall\ 0\lee s\lee r\},$ then $r_{0}\lee r_{*}\lee +\infty$.

\smallskip

Next we aim to show that $r_{*}=+\infty$. Suppose on the contrary that $r_{*}<+\infty$, then $U(r)\gee V(r)$ for any $0\lee r\lee r_{*}$ and $U(r_{*})=V(r_{*})$. Consequently, from
\begin{align}\label{e16}
0&\lee V^{-q}(r)-U^{-q}(r)\lee \Delta^{2}(U-V)(r)\\
\nonumber &=\frac{1}{r^{3}}\left(r^{3}(U-V)^{(3)}\right)'-\frac{1}{r}\left[\Delta(U-V)\right]'(r)\\
\nonumber &=\frac{1}{r}\left(r\left[\Delta(U-V)\right]'\right)', \qquad \forall\ 0\lee r\lee r_{*},
\end{align}
we can derive
\begin{equation}\label{e17}
[\Delta(U-V)]'\gee 0, \qquad \forall\ 0\lee r\lee r_{*}.
\end{equation}
Combining \eqref{e16} and \eqref{e17} yields that
\begin{equation}
(r^{3}(U-V)^{(3)})'(r)\gee 0, \qquad \forall\ 0\lee r\lee r_{*},
\end{equation}
and hence
\begin{equation}\label{e18}
(U-V)^{(k)}(r)\gee 0, \qquad \forall\ 0\lee r\lee r_{*},\ \ \ k=1,2,3.
\end{equation}
It follows from \eqref{18} and $U(r_{*})=V(r_{*})$ that
\begin{equation}
U(r)\equiv V(r), \qquad \forall\ 0\lee r\lee r_{*},
\end{equation}
and hence
\begin{equation}
(U-V)^{(k)}(r)\equiv 0, \ \ \ \ \forall\ 0\lee r\lee r_{*},\ \ \ k=1,2,3.
\end{equation}
Therefore, by entirely similar perturbation method as above, we deduce that there exists a $\tilde{r}_{0}>0$ depending only on $V(r_{*})$ and the function $V(r)$ such that $U(r)\gee V(r)$ for any $r \in [r_{*},r_{*}+\tilde{r}_{0}]$, this contradicts the definition of $r_{*}$. Thus we must have $r_{*}=+\infty$, that is, $U(r)\gee V(r)$ for any $r\gee 0$.

\medskip

Furthermore, if $\Delta^{2}(U-V)(0)>0$, then one has
\begin{equation}
0<\Delta^{2}(U-V)(0)
=\lim_{r\rt 0}{\Delta}^{2}(U-V)(r)
=\frac{8}{3}\lim_{r\rt 0}(U-V)^{(4)}(r)
=\frac{8}{3}(U-V)^{(4)}(0).
\end{equation}
Thus there exists an $\eta_{0}>0$ small enough such that $(U-V)^{(4)}(r)>0$ for all $0\lee r<\eta_{0}$. Therefore, $U(r)-V(r)=\frac{(U-V)^{(4)}(\theta r)}{4!}r^{4}>0$ for any $0<r\lee \eta_{0}$, where $\theta\in(0,1)$.

Set $\bar{r}_{*}:=\sup\{r\ |\ U(s)>V(s),\, \forall\ 0<s\lee r\}$, then $\eta_{0}\lee \bar{r}_{*}\lee +\infty$. Next we will prove $\bar{r}_{*}=+\infty$. Suppose on the contrary that $\bar{r}_{*}< +\infty$, then $U(r)> V(r)$ for all $0<r< \bar{r}_{*}$ and $U(\bar{r}_{*})=V(\bar{r}_{*})$. For any $0< r< \bar{r}_{*}$, $U$ and $V$ satisfy
\begin{align}\label{e19}
0&< V^{-q}(r)-U^{-q}(r)\lee \Delta^{2}\left(U-V\right)(r)\\
\nonumber &=\frac{1}{r^{3}}\left(r^{3}(U-V)^{(3)}\right)'-\frac{1}{r}\left[\Delta\left(U-V\right)\right]'(r)\\
\nonumber &=\frac{1}{r}\left(r\left[\Delta\left(U-V\right)\right]'\right)',
\end{align}
thus we can obtain
\begin{equation}\label{e20}
\left[\Delta(U-V)\right]'(r)> 0, \qquad \forall\ 0< r<\bar{r}_{*}.
\end{equation}
Combining \eqref{e19} and \eqref{e20} implies that
\begin{equation}
\left(r^{3}(U-V)^{(3)}\right)'(r)> 0, \qquad \forall\ 0< r< \bar{r}_{*},
\end{equation}
and hence
\begin{equation}
(U-V)^{(k)}(r)>0, \qquad \forall\ 0< r< \bar{r}_{*},\ \ \ k=1,2,3.
\end{equation}
It follows immediately that
\begin{equation}
(U-V)(\bar{r}_{*})>0,
\end{equation}
which contradicts the definition of $\bar{r}_{*}$. Thus we have proved $\bar{r}_{*}=+\infty$, that is, $U(r)> V(r)$ for any $r>0$. This concludes our proof of Theorem \ref{thm2}.
\end{proof}

We can first derive the following necessary condition on the existence of positive entire solution to the 2D bi-harmonic equation \eqref{PDE}.
\begin{thm}\label{thm3}
Assume $n=2$, $m=2$ and $q>0$. If \eqref{PDE} admits a $C^{4}$ positive solution on entire $\mathbb{R}^{2}$, then $q>1$.
\end{thm}
\begin{proof}
Suppose $u\in C^{4}(\mathbb{R}^{2})$ is a positive entire solution to the bi-harmonic equation \eqref{PDE}. Using the first equation in \eqref{5} and \eqref{9}, we get
\begin{equation}\label{e21}
r\bar{u}'(r)=\int_{0}^{r}t\bar{w}(t)dt \gee \frac{1}{2}r^{2}\bar{w}(r), \qquad \forall \, r\gee 0.
\end{equation}
By dividing \eqref{e21} by $r$ and integrating once again, we get
\begin{equation}\label{13}
\bar{u}(r) \gee u(0) + \frac{1}{4}r^{2}\bar{w}(r), \qquad \forall \, r\gee 0.
\end{equation}

On the other hand, we can also infer from \eqref{PDE}, Lemma \ref{lem0}, Lemma \ref{lem1} and \eqref{8} that
\begin{align}\label{14}
\bar{w}(r)& = \bar{w}(2r)-\int_{r}^{2r}\bar{w}'(t)dt \\
& = \bar{w}(2r)-\frac{1}{2\pi}\int_{r}^{2r}t^{-1}\int_{B_{t}(0)}\Delta{w}dxdt \notag\\
& = \bar{w}(2r)+\frac{1}{2\pi}\int_{r}^{2r}t^{-1}\int_{B_{t}(0)}u^{-q}(x)dxdt \notag\\
& \gee \frac{\ln2}{2\pi}\int_{B_{r}(0)}u^{-q}(x)dx \notag\\
& = \ln 2\int_{0}^{r}t -\kern-13pt\int_{\partial B_{t}(0)}u^{-q}(x)d{\sigma}_{x}dt \notag\\
& \gee \ln 2\int_{0}^{r}t{\bar{u}}^{-q}(t)dt \notag\\
& \gee \frac{\ln 2}{2}r^{2}{\bar{u}}^{-q}(r). \notag
\end{align}
By combining \eqref{14} with \eqref{13}, we conclude that
\begin{equation}
\bar{u}(r) \gee u(0) + \frac{\ln 2}{8}r^{4}{\bar{u}}^{-q}(r),
\end{equation}
and hence
\begin{equation}\label{15}
\bar{u}(r)\gee {\Big{(}\frac{\ln 2}{8}\Big{)}}^{\frac{1}{q+1}}r^{\frac{4}{q+1}}.
\end{equation}
Combining this lower bound \eqref{15} with the upper bound \eqref{7} implies that $4\lee 2(q+1)$. Hence we must have $q\gee 1$.

\medskip

Now suppose that $q=1$ and $u\in C^{4}(\mathbb{R}^{2})$ is a positive entire solution to \eqref{PDE}, we will try to obtain a contradiction. Note that $\bar{u}$ is a sub-solution in the sense that it satisfies $\Delta^{2}\bar{u} + \bar{u}^{-1}\lee 0$, $\bar{u}'(0)=0$ and $\bar{u}'''(0)=0$. Then we define the radially symmetric quadratic function $z(r):=u(0)+\frac{\bar{u}''(0)}{2}r^{2}$. It is a super-solution as it satisfies $\Delta^{2}z+z^{-1}\gee 0$, $z'(0)=0$ and $z'''(0)=0$. Theorem \ref{thm2} (2D comparison theorem) ensures that $z(r)>\bar{u}(r)$ for any $r>0$. Consequently, the radially symmetric solution $U$ of the initial value problem $\Delta^{2}U+U^{-1}=0$ with $U(0)=u(0)$, $U'(0)=0$, $U''(0)=\bar{u}''(0)$ and $U'''(0)=0$ exists on $[0,+\infty)$ and satisfies $\bar{u}\lee U\lee z$ by Theorem \ref{thm2} (2D comparison theorem).

\medskip

Next, let $W(r):=\Delta U(r)$ for $r\gee 0$. With $q=1$, it is immediate that
\begin{equation}\label{e22}
U(r)(rW'(r))'+r=0, \qquad \forall \, r\gee0.
\end{equation}
By integrating \eqref{e22} over the interval $(0,r)$ and integration by parts on the first term twice, we derive
\begin{equation}\label{e23}
U(r)rW'(r)-U'(r)rW(r)+\int_{0}^{r}tW^{2}(t)dt+\frac{1}{2}r^{2}=0, \qquad \forall \, r\gee 0.
\end{equation}
Dividing both sides of \eqref{e23} by $r^{2}$ yields that
\begin{equation}\label{16}
\frac{U(r)}{r^{2}}rW'(r)-\frac{U'(r)}{r}W(r)+\frac{1}{r^{2}}\int_{0}^{r}tW^{2}(t)dt+\frac{1}{2}=0, \qquad \forall \, r>0.
\end{equation}
Thanks to Lemma \ref{lem1} and \eqref{9}, we can obtain that $\lim_{r \rt +\infty}W(r)=\alpha\gee 0$. By \eqref{7} and \eqref{14} with $q=1$, we conclude that $\alpha>0$. Hence we have $(rU')'=\alpha r+o(r)$ for large $r$. By using the L'Hopital's rule, we can reach
\begin{equation}\label{e24}
  \lim_{r \rt +\infty}\frac{U'(r)}{r} = \frac{\alpha}{2} \qquad \text{and} \qquad \lim_{r \rt +\infty}\frac{U(r)}{r^{2}} = \frac{\alpha}{4},
\end{equation}
and hence
\begin{equation}\label{e26}
\lim_{r \rt +\infty}rW'(r) = -\lim_{r \rt +\infty}\int_{0}^{r}tU^{-1}(t)dt=-\infty.
\end{equation}
By using the L'Hopital's rule, we can also arrive at
\begin{equation}\label{e25}
  \lim_{r\rightarrow+\infty}\frac{1}{r^{2}}\int_{0}^{r}tW^{2}(t)dt=\lim_{r\rightarrow+\infty}\frac{W^{2}(r)}{2}=\frac{\alpha^{2}}{2}.
\end{equation}
Combining \eqref{e24}, \eqref{e25} and \eqref{e26}, and letting $r\rightarrow+\infty$ in \eqref{16}, we derive $-\infty+\frac{1}{2}=0$, which is absurd. Therefore, we must have $q>1$. This completes our proof of Theorem \ref{thm3}.
\end{proof}

From Theorem \ref{thm3}, we know that \eqref{PDE} admits no positive entire solution $u\in C^{4}(\mathbb{R}^{2})$ provided that $0<q\lee1$. Therefore, we only need to consider the cases $q>1$ in the rest of this section.
\begin{lem}\label{lem4}
Assume $n=2$, $m=2$ and $q>1$. Suppose $u\in C^{4}(\mathbb{R}^{2})$ is a positive solution to \eqref{PDE}, then
\begin{equation}\label{e31}
  \int_{\mathbb{R}^{2}}u^{-q}(x)dx<+\infty,
\end{equation}
and there exists a constant $C_{0}$ such that
\begin{equation}\label{e32}
  \Delta u(x)=\frac{1}{2\pi}\int_{\mathbb{R}^{2}}\ln\left(\frac{|x-y|}{|y|}\right)u^{-q}(y)dy+C_{0}, \qquad \forall \, x\in\mathbb{R}^{2}.
\end{equation}
Consequently, $\Delta u(x)\lee C\ln|x|$ for $|x|$ sufficiently large.
\end{lem}
\begin{proof}
By \eqref{14}, we have
\begin{equation}\label{e27}
  \int_{B_{r}(0)}u^{-q}(x)dx\lee \frac{2\pi}{\ln2}\bar{w}(r)\lee \frac{2\pi}{\ln2}\bar{w}(0)=\frac{2\pi}{\ln2}\Delta u(0), \qquad \forall \, r\gee 0,
\end{equation}
and hence
\begin{equation}\label{e28}
\int_{\mathbb{R}^{2}}u^{-q}(x)dx<+\infty.
\end{equation}
For arbitrarily fixed $x\in\mathbb{R}^{2}$, there exists $R_{x}>0$ sufficiently large such that
\begin{equation}\label{e29}
  \left|\ln\left(\frac{|x-y|}{|y|}\right)\right|\lee 1, \qquad \forall \, |y|\gee R_{x}.
\end{equation}
As a consequence of \eqref{e28} and \eqref{e29}, we can define the auxiliary function
\begin{equation}\label{20}
v(x):=\frac{1}{2\pi}\int_{\mathbb{R}^{2}}\ln\left(\frac{|x-y|}{|y|}\right)u^{-q}(y)dy, \qquad \forall \, x\in\mathbb{R}^{2}.
\end{equation}
It is easy to verify $\Delta v + u^{-q}=0$ in $\mathbb{R}^{2}$. Thus $\Delta\left(\Delta u-v\right)=0$ in $\mathbb{R}^{2}$. By the definition \eqref{20} of $v(x)$, we can infer from \eqref{e28} that, for $|x|$ sufficiently large,
\begin{align}\label{21}
v(x)&\lee\frac{1}{2\pi}\int_{|y|<2|x|}\ln\left(|x-y|\right)u^{-q}(y)dy-\frac{1}{2\pi}\int_{|y|<1}\ln|y|u^{-q}(y)dy \\
&\quad +\frac{1}{2\pi}\int_{|y|\gee 2|x|}\ln\left(\frac{3}{2}\right)u^{-q}(y)dy \notag \\
&\lee \frac{\ln(3|x|)}{2\pi}\int_{\mathbb{R}^{2}}u^{-q}(y)dy+\frac{\pi}{2}\left[\min_{B_{1}(0)}u\right]^{-q}+o_{|x|}(1) \notag \\
&\lee C\ln|x|. \notag
\end{align}
Consequently, we obtain from Liouville theorem that
\begin{equation}\label{22}
 \Delta u(x)-v(x)=C_{0}, \qquad \forall \, x\in\mathbb{R}^{2},
\end{equation}
and hence, for $|x|$ large enough,
\begin{equation}\label{e30}
  \Delta u(x)=v(x)+C_{0}\lee C\ln|x|.
\end{equation}
This finishes our proof of Lemma \ref{lem4}.
\end{proof}

We will first show that there is no positive entire solution $u\in C^{4}(\mathbb{R}^{2})$ that has a positive lower bound (i.e., $u^{-1}$ is bounded from above).
\begin{thm}\label{thm4}
Assume $n=2$, $m=2$ and $q>1$. Equation \eqref{PDE} admits no positive entire solution $u\in C^{4}(\mathbb{R}^{2})$ satisfying
\begin{equation}\label{e34}
  \lim_{|x|\rightarrow+\infty}\left[\ln |x|\right]^{\frac{1}{q}}u(x)=+\infty.
\end{equation}
Consequently, \eqref{PDE} admits no positive entire solution $u\in C^{4}(\mathbb{R}^{2})$ that has a positive lower bound (i.e., $u^{-1}$ is bounded from above).
\end{thm}
\begin{proof}
Suppose on the contrary that Theorem \ref{thm4} is false, and assume $u(x)\in C^{4}(\mathbb{R}^{2})$ is a positive entire solution to \eqref{PDE} satisfying \eqref{e34}. We will derive a contradiction. By the definition \eqref{20} of $v(x)$, \eqref{e31} and \eqref{e34}, one can verify that, for $|x|$ large enough,
\begin{align}\label{23}
v(x)&\gee\frac{1}{2\pi}\int_{|y|<1}\ln\left(|x-y|\right)u^{-q}(y)dy-\frac{1}{2\pi}\int_{\frac{|x|}{2}\lee |y|<2|x|}\ln|y|u^{-q}(y)dy \\
&\quad +\frac{1}{2\pi}\int_{|y|\gee 2|x|}\ln\left(\frac{1}{2}\right)u^{-q}(y)dy+\frac{1}{2\pi}\int_{|x-y|<1}\ln|x-y|u^{-q}(y)dy \notag \\
& \gee \frac{\ln|x|}{4\pi}\int_{|y|<1}u^{-q}(y)dy-\frac{1}{2\pi}\ln|x|\int_{\frac{|x|}{2}\lee |y|<2|x|}u^{-q}(y)dy \notag \\
& \quad -\frac{\ln 2}{2\pi}\int_{|y|\gee\frac{|x|}{2}}u^{-q}(y)dy-\frac{\pi}{2}\left[\min_{B_{1}(x)}u\right]^{-q} \notag\\
& \gee \frac{\ln|x|}{4\pi}\int_{|y|<1}u^{-q}(y)dy-\left(\frac{1}{2\pi}+\frac{\pi}{2}\right)o\left(\ln|x|\right)-o_{|x|}(1) \notag\\
&\gee c\ln|x|. \notag
\end{align}
Now, we take spherical average with respect to the sphere $\partial B_{r}(0)$ in \eqref{e32} and get
\begin{equation}\label{e35}
  \bar{w}(r)=\bar{v}(r)+C_{0}, \qquad \forall \, r\gee0.
\end{equation}
Letting $r\rightarrow+\infty$ in \eqref{e35}, Lemma \ref{lem2} implies that $\bar{w}(r)\rightarrow c_{\infty}$ for some nonnegative constant $c_{\infty}\gee0$, while \eqref{23} yields that $\bar{v}(r)\gee c\ln r\rt+\infty$. We have reached a contradiction. This completes our proof of Theorem \ref{thm4}.
\end{proof}

The lower bound \eqref{15} can be improved remarkably. In fact, comparing with the upper bound \eqref{7}, we have the following almost sharp lower bound on positive entire solution $u$ to the 2D bi-harmonic equation \eqref{PDE}.
\begin{lem}\label{lem3}
Assume $n=2$, $m=2$ and $q>1$. Suppose $u\in C^{4}(\mathbb{R}^{2})$ is a positive entire solution to \eqref{PDE}, then, for every integer $k\gee 1$, there exists $0<c_{k}\lee \frac{1}{4}\min\limits_{|x|={exp}^{(k)}(1)}\Delta u(x)$ such that
\begin{equation}\label{17}
\bar{u}(r)\gee c_{k}\frac{r^{2}}{{\ln}^{(k)}(r)},  \qquad \forall \, r>0,
\end{equation}
where $\ln^{(k)}:=\underbrace{\ln\cdots\ln}_{k \, \text{times}}$. Consequently, for any $k\gee 1$,
\begin{equation}\label{alb}
  \limsup_{|x|\rt +\infty}\frac{u(x)}{{|x|}^{2}}\ln^{(k)}(|x|)=+\infty.
\end{equation}
\end{lem}
\begin{proof}
For arbitrary $k\in\mathbb{N}^{+}$, set $\delta_{k}=\min\limits_{|x|={\exp}^{(k)}(1)}\Delta u(x)$, where ${\exp}^{(k)}:=\underbrace{\exp\cdots\exp}_{k \, \text{times}}$. By Lemma \ref{lem1}, one has $\Delta u>0$ and hence $\delta_{k}>0$. Applying the maximum principle to the function $\frac{\delta_{k}}{{\ln}^{(k)}(|x|)}-\Delta u$ on the region $\{x\in \mathbb{R}^{2}| \, {\exp}^{(k)}(1) \lee |x|<+\infty\}$, we obtain that
\begin{equation}\label{e33}
  \Delta u(x)\gee \frac{\delta_{k}}{{\ln}^{(k)}(|x|)}, \qquad \forall \, |x| \gee {\exp}^{(k)}(1), \quad \forall \, k\gee 1.
\end{equation}
Combining this with \eqref{13}, we have, for all $k\in\mathbb{N}^{+}$ and $r$ sufficiently large,
\begin{equation}\label{17'}
\bar{u}(r) \gee \frac{r^{2}}{4}\frac{\delta_{k}}{{\ln}^{(k)}(r)}.
\end{equation}
Therefore, for any $k\gee 1$, there exists $0<c_{k}\lee\frac{\delta_{k}}{4}$ such that
\begin{equation}\label{17+}
\bar{u}(r) \gee c_{k}\frac{r^{2}}{{\ln}^{(k)}(r)}, \qquad \forall \, r>0.
\end{equation}
Due to the arbitrariness of $k\gee 1$, \eqref{17'} immediately leads to
\begin{equation}\label{18}
\limsup_{|x|\rt+\infty}\frac{u(x)}{{|x|}^{2}}{\ln}^{(k)}(|x|)=+\infty, \qquad \forall \, k\gee1.
\end{equation}
This finishes our proof of Lemma \ref{lem3}.
\end{proof}

Lemma \ref{lem3} indicates that any positive entire solution $u\in C^{4}(\mathbb{R}^{2})$ to \eqref{PDE} has almost quadratic growth (in the sense of spherical average) at $\infty$. As a direct consequence of Theorem \ref{thm4} and Lemma \ref{lem3}, we can first deduce immediately the nonexistence of radially symmetric positive entire solutions to \eqref{PDE}, which already gives a negative answer to the open question (2) raised by Mckenna and Reichel in Section 6 of \cite{MR}.
\begin{thm}\label{thm5}
Assume $n=2$, $m=2$ and $q>1$. Equation \eqref{PDE} admits no radially symmetric positive entire solution $u\in C^{4}(\mathbb{R}^{2})$.
\end{thm}
\begin{proof}
Suppose on the contrary that $u\in C^{4}(\mathbb{R}^{2})$ is a radially symmetric positive entire solution to \eqref{PDE}. Then we can deduce immediately from \eqref{7} and \eqref{17} in Lemma \ref{lem3} that, there exist $C>0$ and $c_{k}>0$ ($k\gee1$) such that, for every integer $k\gee 1$,
\begin{equation}\label{17-rs}
c_{k}\frac{|x|^{2}}{{ln}^{(k)}\left(|x|\right)}\lee u(x)\lee C|x|^{2},  \qquad \forall \, |x| \,\, \text{large enough}.
\end{equation}
Hence, by Theorem \ref{thm4}, we arrive at a contradiction. This finishes our proof of Theorem \ref{thm5}.
\end{proof}

\noindent\textbf{Completion of our proof of Theorem \ref{thm0}.} Now we will deduce the absence of positive entire solution to \eqref{PDE} from the nonexistence of positive solution that has positive lower bound (i.e., $u^{-1}$ has upper bound) in Theorem \ref{thm4} by using the doubling lemma from \cite{PQS}, and hence complete our proof of Theorem \ref{thm0}.

\begin{lem}[Doubling lemma, Lemma 5.1 in \cite{PQS}]\label{doubling}
Let $(X,d)$ be a complete metric space and let $\emptyset\neq D\subset\Sigma\subset X$ with $\Sigma$ closed. Set $\Gamma=\Sigma\setminus D$. Finally let $M:\,D\rightarrow(0,+\infty)$ be bounded on compact subsets of $D$ and fix a $k>0$. If $y\in D$ is such that
\begin{equation}\label{d-1}
  M(y)dist(y,\Gamma)>2k,
\end{equation}
then there exists $x\in D$ such that
\begin{equation}\label{d-2}
   M(x)dist(x,\Gamma)>2k, \qquad M(x)\gee M(y),
\end{equation}
and
\begin{equation}\label{d-3}
  M(z)\lee 2M(x), \qquad \forall \, z\in \overline{B_{\frac{k}{M(x)}}\left(x\right)}\subset X.
\end{equation}
\end{lem}

Suppose on the contrary that Theorem \ref{thm0} is false and $u\in C^{4}(\mathbb{R}^{2})$ is a positive entire solution to \eqref{PDE} satisfying $\inf_{\mathbb{R}^{2}}u=0$, we will derive a contradiction from Theorem \ref{thm4}. We take $D=\Sigma=X=\mathbb{R}^{2}$ and $M(x):=u(x)^{-\frac{q+1}{4}}$ for any $x\in\mathbb{R}^{2}$ in Lemma \ref{doubling}. Then $\Gamma=\Sigma\setminus D=\emptyset$ and $dist(x,\Gamma)=+\infty$ for all $x\in X$, and there exists a sequence $\{y_{k}\}\subset\mathbb{R}^{2}$ such that $M(y_{k})=u(y_{k})^{-\frac{q+1}{4}}\rightarrow+\infty$. By the doubling lemma (Lemma \ref{doubling}), there exists a sequence $\{x_{k}\}\subset\mathbb{R}^{2}$ such that, for every $k=1,2,\cdots$,
\begin{equation}\label{e36}
  M(x_{k})\gee M(y_{k}),  \qquad M(z)\lee 2M(x_{k}), \quad \forall \, z\in \overline{B_{\frac{k}{M(x_{k})}}(x_{k})}.
\end{equation}
For every $k=1,2,\cdots$, define
\begin{equation}\label{e37}
  u_{k}(x):=\frac{u\left(x_{k}+\frac{x}{M(x_{k})}\right)}{u(x_{k})}, \qquad \forall \, x\in\mathbb{R}^{2}.
\end{equation}
Then $u_{k}\in C^{4}(\mathbb{R}^{2})$ is a positive entire solution to \eqref{PDE} such that $u_{k}(0)=1$ and $u_{k}(x)\gee \left(\frac{1}{2}\right)^{\frac{4}{q+1}}$ in $\overline{B_{k}(0)}$. By equation \eqref{PDE} and $\|u_{k}^{-q}\|_{L^{\infty}(\overline{B_{k}(0)})}\lee 2^{\frac{4q}{q+1}}$, we can infer from regularity theory (see e.g. Corollary 7 in \cite{RW}) that $\{u_{k}\}$ is locally $W^{4,p}$ bounded on $\mathbb{R}^{2}$ for any $p\in(1,+\infty)$, and hence by Sobolev embedding,
\begin{equation}\label{32-6}
  \|u_{k}\|_{C^{3,\gamma}(\overline{B_{1}(0)})}\lee C, \qquad \forall \, k\gee 1,
\end{equation}
where $0\lee\gamma<1$ and $C$ is independent of $k$. Furthermore, by equation \eqref{PDE}, regularity theory and Sobolev embedding again, one has
\begin{equation}\label{32-7}
   \|u_{k}\|_{C^{6,\gamma}(\overline{B_{1}(0)})}\lee C
\end{equation}
for every $k\gee 1$, where $0\lee\gamma<1$ and $C$ is independent of $k$. As a consequence, by Arzel\`{a}-Ascoli Theorem, there exists a subsequence $\{u^{(1)}_{k}\}\subset\{u_{k}\}$ and a function $\hat{u}\in C^{4}(\overline{B_{1}(0)})$ such that
\begin{equation}\label{32-8}
  u^{(1)}_{k}\rightrightarrows \hat{u} \quad \text{and} \quad (-\Delta)^{2}u^{(1)}_{k}\rightrightarrows(-\Delta)^{2}\hat{u} \quad\quad \text{in} \,\, \overline{B_{1}(0)}.
\end{equation}
By equation \eqref{PDE} and $\|u_{k}^{-q}\|_{L^{\infty}(\overline{B_{k}(0)})}\lee 2^{\frac{4q}{q+1}}$, we can deduce from regularity theory and Sobolev embedding again that
\begin{equation}\label{32-9}
  \|u^{(1)}_{k}\|_{C^{6,\gamma}(\overline{B_{2}(0)})}\lee C
\end{equation}
for every $k\gee 2$, where $0\lee\gamma<1$. Therefore, by Arzel\`{a}-Ascoli Theorem again, there exists a subsequence $\{u^{(2)}_{k}\}\subset\{u^{(1)}_{k}\}$ and $\hat{u}\in C^{4}(\overline{B_{2}(0)})$ such that
\begin{equation}\label{32-10}
  u^{(2)}_{k}\rightrightarrows \hat{u} \quad \text{and} \quad (-\Delta)^{2}u^{(2)}_{k}\rightrightarrows(-\Delta)^{2}\hat{u} \quad\quad \text{in} \,\, \overline{B_{2}(0)}.
\end{equation}
Continuing this way, for any $j\in\mathbb{N}^{+}$, we can extract a subsequence $\{u^{(j)}_{k}\}\subset\{u^{(j-1)}_{k}\}$ and find a function $\hat{u}\in C^{4}(\overline{B_{j}(0)})$ such that
\begin{equation}\label{32-11}
  u^{(j)}_{k}\rightrightarrows \hat{u} \quad \text{and} \quad (-\Delta)^{2}u^{(j)}_{k}\rightrightarrows(-\Delta)^{2}\hat{u} \quad\quad \text{in} \,\, \overline{B_{j}(0)}.
\end{equation}
By extracting the diagonal sequence, we finally obtain that the subsequence $\{u^{(k)}_{k}\}$ satisfies
\begin{equation}\label{32-12}
  u^{(k)}_{k}\rightrightarrows \hat{u} \quad \text{and} \quad (-\Delta)^{2}u^{(k)}_{k}\rightrightarrows(-\Delta)^{2}\hat{u} \quad\quad \text{in} \,\, \overline{B_{j}(0)}
\end{equation}
for any $j\gee 1$. Therefore, $\hat{u}\in C^{4}(\mathbb{R}^{2})$ satisfies
\begin{equation}\label{32-13}
  (-\Delta)^{2}\hat{u}(x)+\hat{u}^{-q}(x)=0 \quad\quad \text{in} \,\, \mathbb{R}^{2}
\end{equation}
with $\hat{u}(0)=1$ and $\inf\limits_{\mathbb{R}^{2}}\hat{u}\gee\left(\frac{1}{2}\right)^{\frac{4}{q+1}}$. That is, $\hat{u}$ is a positive entire solution of \eqref{PDE} satisfying $\inf\limits_{\mathbb{R}^{2}}\hat{u}\gee\left(\frac{1}{2}\right)^{\frac{4}{q+1}}$, which contradicts Theorem \ref{thm4}. This concludes our proof of Theorem \ref{thm0}.

\end{document}